\documentclass[12pt]{amsart}
\usepackage{amsfonts}
\usepackage{latexsym}
\usepackage{amssymb}
\usepackage{amsmath}
\usepackage{color}
\usepackage{bbm}
\usepackage{tikz}
\usepackage{enumerate}
\usepackage{mathrsfs}
\usepackage{todonotes}

\usepackage[left=2.5cm, top=2.5cm,bottom=2.5cm,right=2.5cm]{geometry}

\newcommand{\asymptequiv}{\sim}

\newcommand{\F}{\mathbb F}
\newcommand{\Ham}{\mathbb H}
\newcommand{\R}{\mathbb R}
\newcommand{\N}{\mathbb N}
\newcommand{\C}{\mathbb C}
\newcommand{\E}{\mathbb E}
\newcommand{\Pro}{\mathbb P}

\newcommand{\vol}{\mathrm{vol}}
\newcommand{\Mat}{\mathrm{Mat}}

\def\dint{\textup{d}}
\newcommand{\SSS}{\ensuremath{{\mathbb S}}}

\newcommand{\B}{\ensuremath{{\mathbb B}}}
\newcommand{\D}{\ensuremath{{\mathbb D}}}

\newcommand{\eps}{\varepsilon}

\DeclareMathOperator{\sgn}{sgn}
\DeclareMathOperator{\Tr}{Tr}

\renewcommand{\Re}{\operatorname{Re}}  

\newtheorem{thm}{Theorem}[section]
\newtheorem{cor}[thm]{Corollary}
\newtheorem{lemma}[thm]{Lemma}

{
\theoremstyle{definition}

\newtheorem{rmk}[thm]{Remark}

}

\usepackage{nomencl}
\makenomenclature

\begin{document}


\title[]{Intersection of unit balls in \\classical matrix ensembles}

\author[Z. Kabluchko]{Zakhar Kabluchko}
\address{Zakhar Kabluchko: Institut f\"ur Mathematische Stochastik, Westf\"alische Wilhelms-Uni\-ver\-sit\"at M\"unster, Germany}
\email{zakhar.kabluchko@uni-muenster.de}

\author[J. Prochno]{Joscha Prochno}
\address{Joscha Prochno: School of Mathematics \& Physical Sciences, University of Hull, United Kingdom} \email{j.prochno@hull.ac.uk}

\author[C. Th\"ale]{Christoph Th\"ale}
\address{Christoph Th\"ale: Fakult\"at f\"ur Mathematik, Ruhr-Universit\"at Bochum, Germany} \email{christoph.thaele@rub.de}




\begin{abstract}
We study the volume of the intersection of two unit balls from one of the classical matrix ensembles GOE, GUE and GSE, as the dimension tends to infinity. This can be regarded as a matrix analogue of a result of Schechtman and Schmuckenschl\"ager for classical $\ell_p$-balls [Schechtman and Schmuckenschl\"ager, GAFA Lecture Notes, 1991]. The proof of our result is based on two ingredients, which are of independent interest. The first one is a weak law of large numbers for a point chosen uniformly at random in the unit ball of such a matrix ensemble. The second one is an explicit computation of the asymptotic volume of such matrix unit balls, which in turn is based on the theory of logarithmic potentials with external fields.
\end{abstract}

\maketitle

\tableofcontents

\section{Introduction and main result}

To understand the geometry of high-dimensional convex bodies and, in particular, the distribution of volume is one of the central aspects considered in Asymptotic Geometric Analysis. It has been realized by now that such an understanding has important connections and implications to various questions considered in other branches of mathematics and related disciplines. We refer the reader to the research monographs and surveys \cite{AsymptoticGeometricAnalysisBookPart1, IsotropicConvexBodies,G2014b, G2014a} for background information.

Ever since, there has been a particular interest and focus on the non-commutative setting of Schatten trace classes or classical matrix ensembles as is demonstrated by the research carried out in \cite{CK2015, GP2007, HPV17, KMP1998, RV2016, ST1980}. While this often underlines a similarity to the commutative setting of classical $\ell_p$ sequence spaces, it also shows differences in the behavior of certain quantities related to the geometry of Banach spaces. In fact, often different methods and tools are needed and proofs can be considerably more involved.

In the classical setting of $\ell_p^n$-balls, Schechtman and Zinn \cite{SchechtmanZinn} considered the question of what proportion of volume is left in a volume-normalized $\ell^n_p$-ball after removing a $t$-multiple of a volume-normalized $\ell_q^n$-ball. In the case $p=1$ and $q=2$ this question was raised by V.D. Milman.
In a paper subsequent to \cite{SchechtmanZinn}, Schechtman and Schmuckenschl\"ager \cite{SchechtmanSchmuckenschlaeger} investigated the \emph{asymptotic} behavior of the volume of such intersections. More precisely, if we denote by $\D_p^n$ the $\ell_p^n$-ball of unit volume, then, using a law of large numbers, Schechtman and Schmuckenschl\"ager proved that if $1\leq p,q\leq\infty$ are such that $q\neq p$ and $q<\infty$, then
\begin{align}\label{eq:SchechtmanSchmuckenschlaeger}
\vol_n\big( \D_p^n \cap t\D_q^n\big)
\stackrel{n\to\infty}{\longrightarrow}
\begin{cases}
0 &: tA_{p,q}<1\\
1 &: tA_{p,q}>1
\end{cases}
\end{align}
for all $t\geq 0$. Here, the constant $A_{p,q}$ is given as follows:
$$
A_{p,q} =  \begin{cases}
\frac {\Gamma(1+{1\over p})^{1+{1/q}}}{\Gamma(1+{1\over q})\Gamma({q+1\over p})^{1/q}}\,e^{{1/p}-{1/q}}\,\big({p\over q}\big)^{1/q} &: p<\infty\\
\Gamma(1+{1\over q})^{-1} \big(\frac {q+1}{qe}\big)^{1/q} &: p=\infty.
\end{cases}
$$
The critical case where $t A_{p,q} = 1$ has later been handled by Schmuckenschl\"ager \cite{schmuckenschlaeger_pams,SchmuckenschlaegerCLT} using a central limit theorem. We also refer to \cite{KPT17CCM} for multivariate analogues and some new developments in this direction.

The purpose of the present paper is to establish a non-commutative analogue to \eqref{eq:SchechtmanSchmuckenschlaeger} for the unit balls of different classical matrix ensembles. More precisely, we let $\beta\in\{1,2,4\}$ and consider the collection $\mathscr H_n(\mathbb{F}_\beta)$ of all self-adjoint $n\times n$ matrices with entries from the skew field $\mathbb{F}_\beta$, where $\mathbb{F}_1=\R$, $\mathbb{F}_2=\C$ or $\mathbb{F}_4=\Ham$, the set of Hamiltonian quaternions, see Section \ref{sec:Preliminaries} for more details. The standard Gaussian distribution on $\mathscr H_n(\mathbb{F}_\beta)$ is known as the Gaussian orthogonal ensemble (GOE) if $\beta=1$, the Gaussian unitary ensemble (GUE) if $\beta=2$, or the Gaussian symplectic ensemble if $\beta=4$.

By $\lambda_1(A),\ldots,\lambda_n(A)$ we denote the (real) eigenvalues of a matrix $A$ from $\mathscr H_n(\mathbb{F}_\beta)$ and consider the following matrix analogues of the classical $\ell_p^n$-balls discussed above:
$$
\B_{p,\beta}^n := \Big\{A\in \mathscr H_n(\mathbb{F}_\beta):\sum_{j=1}^n|\lambda_j(A)|^p \leq 1\Big\},\qquad \beta \in\{1,2,4 \}\quad\text{and}\quad  0 < p \leq \infty,
$$
where we interpret the defining condition in brackets as $\max\{|\lambda_j(A)|:j=1,\ldots,n\} \leq 1$ if $p=\infty$. As in the case of the classical $\ell_p^n$-balls, we denote by $\D_{p,\beta}^n$, $\beta\in\{1,2,4\}$ the volume normalized versions of these matrix unit balls (the volume in $\mathscr H_n(\mathbb F_\beta)$ will formally be introduced in Section \ref{sec:Preliminaries} below). In this paper we prove the following matrix analogue to \eqref{eq:SchechtmanSchmuckenschlaeger}.

\begin{thm}\label{thm:ApplInto}
Let $0 <  p, q <\infty$ with $p\neq q$ and $\beta\in\{1,2,4\}$. Then, for $t>0$,
\[
\vol(\D^n_{p,\beta}\cap t\, \D^n_{q,\beta}) \stackrel{n\to\infty}{\longrightarrow}
\begin{cases}
0 &: t < e^{\frac{1}{2p} - \frac{1}{2q}} \big(\frac{2p}{p+q}\big)^{1/q}\\
1 &: t > e^{\frac{1}{2p} - \frac{1}{2q}} \big(\frac{2p}{p+q}\big)^{1/q}\,.
\end{cases}
\]
\end{thm}

Let us briefly comment on the similarities and differences between \eqref{eq:SchechtmanSchmuckenschlaeger} and the result of Theorem \ref{thm:ApplInto}. While the structural statements are the same, the thresholds are significantly different. The fact that the constant in Theorem \ref{thm:ApplInto} is considerably more simple than $A_{p,q}$ in \eqref{eq:SchechtmanSchmuckenschlaeger} can roughly be explained as follows: to quantify whether a point from $\D_{p,\beta}^n$ also belongs to $t\D_{q,\beta}^n$ finally boils down to a moment comparison of a so-called Ullman random variable (and a different random element in case of the classical $\ell_p^n$-balls). While in the classical case, this ratio essentially corresponds to $A_{p,q}$, in the matrix set-up this expression simplifies considerably, since the terms involving gamma functions finally cancel out as they do not depend on $p$ and $q$ simultaneously, in contrast to the classical set-up.

Let us emphasize that the proof of Theorem \ref{thm:ApplInto}  is considerably more involved than its $\ell_p^n$-ball counterpart in \cite{SchechtmanSchmuckenschlaeger}. It is essentially based on two results that are of independent interest. The first result is a precise description of the asymptotic volume of the matrix balls $\B_{p,\beta}^n$, as $n\to\infty$. While such a result is known up to a \emph{non-explicit} constant from the work of Saint Raymond~\cite{R1984} for the unit balls of Schatten classes, where the matrices are not self-adjoint and the eigenvalues are replaced by the singular values, this seems (surprisingly) not to be case for the matrix balls $\B_{p,\beta}^n$. We refer to the discussion in \cite{GP2007} where also asymptotic lower and upper bounds for the volume of $\B_{p,\beta}^n$ have been derived with non-explicit constants. However, we emphasize that for our purposes the explicit asymptotic constants are in fact needed. The proof of an explicit asymptotic formula for the volume of $\B_{p,\beta}^n$, namely
$$
\lim_{n\to \infty}
n^{{1\over p}+{1\over 2}}
\vol (\B_{p,\beta}^n)^{2/(\beta n^2)}
=
\bigg(\frac{p\sqrt{\pi}\,\Gamma(\frac{p}{2})}{\sqrt{e}\,\Gamma(\frac{p+1}{2})}\bigg)^{\frac{1}{p}}
\left({\pi\over \beta}\right)^{\frac{1}{2}}e^{\frac{3}{4}},
$$
is the content of Section \ref{sec:MatrixBallsVolume}. In this context we would like to emphasize that already the asymptotic volume formula for the unit balls of Schatten classes in \cite{R1984} contained a certain non-explicit factor whose analogue in our set-up is denoted by $\Delta(p)$. While only lower and upper bounds for the factor appearing in~\cite{R1984} are known, we shall provide an explicit formula for $\Delta(p)$. For the proof we deploy results from the theory of logarithmic potentials with external fields. In this spirit our analysis sharpens the result in \cite{R1984} and at the same time we are aiming to make more transparent the proof and its essential elements. We shall handle in detail the case of unit balls in Schatten $p$-classes in a parallel paper. The second ingredient of the proof of Theorem \ref{thm:ApplInto} is a weak law of large numbers for the eigenvalues of a matrix uniformly distributed in $\B_{p,\beta}^n$. This result in turn will be a consequence of a Schechtman-Zinn-type probabilistic representation of the volume measure on $\B_{p,\beta}^n$, which we derive from the classical polar integration formula for the cone measure. On the other hand, it is based on a limit theorem from random matrix theory about the empirical eigenvalue distribution of general $\beta$-ensembles.

\medspace

The remaining parts of the paper are organized as follows. In Section \ref{sec:Preliminaries} we present the notation and introduce the tools required to prove our main results. Since some key elements of the proofs are not common in the theory of asymptotic geometric analysis so far, we shall introduce them in slightly more detail. The computation of the asymptotic volume of unit balls in classical matrix ensembles is treated in Section \ref{sec:MatrixBallsVolume}. The probabilistic elements, in particular the weak law of large numbers for the eigenvalues of a matrix chosen uniformly at random from $\B_{p,\beta}^n$, are part of Section \ref{sec:sampling and weak law}. In the final Section \ref{sec:Application}, we shall address the question of what proportion of volume is left in a volume-normalized ball $\B_{p,\beta}^n$ after removing a $t$-multiple of a volume-normalized ball $\B_{q,\beta}^n$.

\section{Preliminaries}\label{sec:Preliminaries}

\subsection{Some general notation}

We let $\R$ be the set of real numbers, $\C$ be the set of complex numbers with standard basis $\{1,\mathfrak{i}\}$ and further $\Ham$ the set of Hamiltonian quaternions with standard basis denoted by $\{1,\mathfrak{i},\mathfrak{j},\mathfrak{k}\}$. We define for $\beta\in\{1,2,4\}$ the (skew) field
$$
\F_\beta := \begin{cases}
\R &: \beta =1\\
\C &: \beta =2\\
\Ham &:\beta = 4,
\end{cases}
$$
and notice that $\beta$ is the dimension of $\F_\beta$ over $\R$.

By $\R^n$ we denote the $n$-dimensional Euclidean space and write $\langle\, \cdot\,,\,\cdot\,\rangle$ for its standard inner product. For a topological space $E$ we shall write $\mathscr{B}(E)$ for the Borel $\sigma$-field on $E$. The $n$-volume (i.e.,\ $n$-dimensional Lebesgue measure) of a Borel set $A\in\mathscr{B}(\R^n)$ will be denoted by $\vol_n(A)$. For $0< p \leq \infty$ and $x=(x_1,\ldots,x_n)\in\R^n$ we let $\|x\|_p$ be the $p$-norm of $x$ (which, in fact, is only a quasi-norm if $p<1$) given by
$$
\|x\|_p := \begin{cases}
\Big(\sum\limits_{i=1}^n|x_i|^p\Big)^{1/p} &: p<\infty\\
\max\{|x_1|,\ldots,|x_n|\} &: p=\infty.
\end{cases}
$$
We write $\B_p^n:=\{x\in\R^n:\|x\|_p\leq 1\}$ and $\SSS_p^{n-1}:=\{x\in\R^n:\|x\|_p=1\}$ for the unit ball and the unit sphere with respect to the $p$-norm in $\R^n$.

The cone probability measure $\mu_{\B_p^n}$ on $\SSS_p^{n-1}$ is defined as
$$
\mu_{\B_p^n}(A) := {1\over\vol_n(\B_p^n)}\vol_n\big(\{rx:r\in[0,1],x\in A\}\big),\qquad A\in\mathscr{B}(\SSS_p^{n-1}).
$$
We remark that $\mu_{\B_p^n}$ coincides with the corresponding normalized Hausdorff measure on $\SSS_p^{n-1}$ if and only if $p\in\{1,2,\infty\}$ (see, e.g., \cite{NaorTAMS}). The cone measure may alternatively be defined as the (unique) measure satisfying the polar integration formula
	\begin{align}\label{eq:polar integration}
	\int_{\R^n}f(x)\,\dint x = n\,\vol_n(\B_p^n)\int_0^\infty\int_{\SSS_p^{n-1}}f(ry)\,r^{n-1}\,\mu_{\B_p^n}(\dint y)\,\dint r
	\end{align}
for all non-negative and Borel measurable functions $f:\R^n\to\R$ (see, e.g., \cite[Proposition 1]{NR2003}). We remark that a similar formula holds for general star-shaped bodies in $\R^n$.

We shall denote by $\mathfrak{S}(n)$ the group of permutations on the set $\{1,\ldots,n\}$. If a constant depends on a parameter such as $\beta$ and/or $p$ we shall indicate this by lower indices, i.e., by writing $C_{\beta}$ or $C_{\beta,p}$. Finally, we frequently use for sequences $(a_n)_{n\in\N}$ and $(b_n)_{n\in\N}$ the asymptotic notation $a_n \asymptequiv b_n$ to indicate that $\frac{a_n}{b_n}\to 1$, as $n\to\infty$.

\subsection{Random measures}

Let $S$ be a Polish space and $(\Omega,\mathcal{A},\Pro)$ be a probability space, which we implicitly assume to be rich enough to carry all the random elements we consider. A measure $\mu$ on $S$ is said to be locally finite if $\mu(B)<\infty$ for all bounded Borel sets $B\in\mathscr{B}(S)$. We denote by $\mathcal{M}_S$ the space of locally finite measures on $S$ and supply $\mathcal{M}_S$ with the $\sigma$-field $\mathscr{B}(\mathcal{M}_S)$ generated by the evaluation mappings $e_B:\mu\mapsto\mu(B)$, where $\mu\in\mathcal{M}_S$ and $B\in\mathscr{B}(S)$, i.e., $\mathscr{B}(\mathcal{M}_S)$ is the smallest $\sigma$-field for which all the mappings $e_B$ become measurable.
We remark that with the vague topology, i.e., the topology generated by the mappings
\[
\mathcal{M}_S\to\R,\qquad \mu\mapsto \int_S f \,\dint \mu
\]
with $f$ being some continuous and compactly supported function on $S$, the space $\mathcal{M}_S$ is Polish and $\mathscr{B}(\mathcal{M}_S)$ is its associated Borel $\sigma$-field, see \cite[Theorem 4.2]{KallenbergRM}. By a random measure $\xi$ on $S$ we understand a random element in the measurable space $\mathcal{M}_S$, i.e., a measurable mapping $\xi:(\Omega,\mathcal A)\to (\mathcal M_S,\mathscr B(\mathcal M_S))$.

Let $\xi_1,\xi_2,\ldots$ be a sequence of random measures and $\xi$ be another random measure. We say that $(\xi_n)_n$ converges weakly almost surely (or weakly with probability one) to $\xi$, provided that there exists $\Omega_0\subset\Omega$ with $\Pro(\Omega_0)=1$ such that the weak convergence $\xi_n(\omega)\overset{w}{\longrightarrow}\xi(\omega)$ holds for all $\omega\in\Omega_0$.

For further background material on random measure theory we refer to the monograph \cite{KallenbergRM}.

\subsection{Convergence results from probability theory}

In our arguments below we need a couple of convergence results from probability theory. They are well known, but having a broad readership in mind we also include their short proofs. The first lemma connects convergence in distribution with convergence in probability of a sequence of random variables. For random variables $X,X_1,X_2,\ldots$ we write $X_n\overset{\text{d}}{\longrightarrow}X$ to indicate that $X_n$ converges to $X$ in distribution, and $X_n\overset{\text{a.s.}}{\longrightarrow}X$ or $X_n\overset{\Pro}{\longrightarrow}X$ if $X_n$ converges to $X$ almost surely or in probability, as $n\to\infty$, respectively.

\begin{lemma}\label{lem:ConvergenceDistributionProbability}
Let $X_1,X_2,\ldots$ be real-valued random variables and $c\in\R$ be a constant. Assume that $X_n\overset{\text{d}}{\longrightarrow}c$, as $n\to\infty$. Then $X_n\overset{\Pro}{\longrightarrow}c$, as $n\to\infty$.
\end{lemma}
\begin{proof}
Fix $\varepsilon>0$ and note that
$$
\Pro(|X_n-c|>\varepsilon) = \Pro(X_n\notin (c-\varepsilon,c+\varepsilon)).
$$
By the portmanteau theorem \cite[Theorem 4.25 (iii)]{Kallenberg} we have that, since $X_n\overset{\text{d}}{\longrightarrow}c$ and since the complement of $(c-\varepsilon,c+\varepsilon)$ is a closed subset of $\R$,
$$
\limsup_{n\to\infty}\Pro(X_n\notin (c-\varepsilon,c+\varepsilon)) \leq \Pro(c\notin(c-\varepsilon,c+\varepsilon)) = 0.
$$
Therefore,
$$
\lim_{n\to\infty}\Pro(|X_n-c|>\varepsilon) = 0,
$$
meaning that $X_n\overset{\Pro}{\longrightarrow}c$, as $n\to\infty$.
\end{proof}

The next Slutsky-type result deals with convergence in probability of products and quotients.

\begin{lemma}\label{lem:ConvergenceProductsQuotients}
Let $X,X_1,X_2,\ldots$ and $Y,Y_1,Y_2,\ldots$ be real-valued random variables. Assume that $X_n\overset{\Pro}{\longrightarrow}X$ and $Y_n\overset{\Pro}{\longrightarrow}Y$, as $n\to\infty$.
\begin{itemize}
\item[(i)]  One has that $X_nY_n\overset{\Pro}{\longrightarrow}XY$, as $n\to\infty$.
\item[(ii)] Suppose in addition that $\Pro(Y=0)=0$. Then $X_n/Y_n\overset{\Pro}{\longrightarrow}X/Y$, as $n\to\infty$.
\end{itemize}
\end{lemma}
\begin{proof}
To prove (i) we note that since $X_n\overset{\Pro}{\longrightarrow}X$ and $Y_n\overset{\Pro}{\longrightarrow}Y$ we also have that $(X_n,Y_n)\overset{\Pro}{\longrightarrow}(X,Y)$, as $n\to\infty$. Applying the continuous mapping theorem \cite[Lemma 3.3]{Kallenberg} to the continuous function $f:\R^2\to\R,(x,y)\mapsto xy$ we conclude that $X_nY_n\overset{\Pro}{\longrightarrow}XY$ and the proof is complete.

To prove (ii) we let $D_f$ be the set of discontinuity points of a function $f:\R\to\R$. Since $\Pro(Y=0)=0$ we have that $\Pro(Y\in D_f)=0$ for the function $f(x)=1/x$. Whence, $1/Y_n\overset{\Pro}{\longrightarrow}1/Y$, as $n\to\infty$, by the continuous mapping theorem \cite[Lemma 3.3]{Kallenberg}. The result follows now from part (i).
\end{proof}

\subsection{Gaussian ensembles}

We denote for $n\in\N$ and $\beta\in\{1,2,4\}$ by $\Mat_{n}(\mathbb F_\beta)$ the space of $n\times n$ matrices with entries from $\F_\beta$. For a matrix $A\in \Mat_{n}(\mathbb F_\beta)$ we let $A^*$ be the adjoint of $A$, i.e.\  the matrix obtained from $A$ first by transposing $A$ and then applying the conjugation operation to each entry. We are interested in the matrix spaces
$$
\mathscr H_n(\mathbb{F}_\beta) := \{A\in \Mat_{n}(\mathbb F_\beta):A=A^*\}.
$$
Clearly, each $\mathscr H_n(\mathbb{F}_\beta)$ is a vector space over $\R$. Endowed with the scalar product $\langle A, B \rangle = \Re \Tr (A B^*)$, where $\Tr$ is the trace of a matrix, $\mathscr H_n(\mathbb{F}_\beta)$ becomes a Euclidean space. We denote by $\vol_{\beta,n}(\,\cdot\,)$ the (Riemannian) volume measure on $\mathscr H_n(\mathbb{F}_\beta)$ corresponding to this scalar product. Let us remark that this measure coincides with the (suitably normalized) $({\beta n(n-1)\over 2}+n)$-dimensional Hausdorff measure on $\mathscr H_n(\mathbb{F}_\beta)$ as follows directly from the area-coarea formula.

For each self-adjoint matrix $A\in \mathscr H_n(\mathbb{F}_\beta)$ we denote by $\lambda_1(A),\ldots,\lambda_n(A)$ the (real) eigenvalues of $A$ and simply write $\lambda_1,\dots,\lambda_n$ if it is unambiguous what the underlying matrix is. We refer to Appendix E in \cite{AGZ2010} for a formal definition of eigenvalues in the symplectic case, i.e., when $\beta=4$. Let us further define the constant
\begin{equation}\label{eq:DefConstantcnbeta}
c_{n,\beta} := {1\over n!}\bigg({2\pi^{\beta/2}\over\Gamma({\beta\over 2})}\bigg)^{-1}\,{\prod\limits_{k=1}^n{2(2\pi)^{\beta k/2}\over 2^{\beta/2}\Gamma({\beta k\over 2})}}\,.
\end{equation}
This allows us to recall the following Weyl integration formula that can be found in \cite[Proposition 4.1.1]{AGZ2010}; see also \cite[Proposition 4.1.14]{AGZ2010} from which the formula for $c_{n,\beta}$ can be derived.

\begin{lemma}\label{lem:IntegrationFormulaForEnsembles}
Fix $n\in\N$ and $\beta\in\{1,2,4\}$. Let $f:\mathscr H_n(\mathbb{F}_\beta)\to\R$ be a non-negative and Borel measurable function such that $f(A)$ only depends on the eigenvalues of $A$. Then
$$
\int_{\mathscr H_n(\mathbb{F}_\beta)} f(A)\,\vol_{\beta,n}(\dint A) = c_{n,\beta}\int_{\R^n} f(\lambda)\,\prod_{1\leq i<j\leq n}|\lambda_i-\lambda_j|^\beta\,
\dint \lambda_1 \ldots \dint \lambda_n\,,
$$
where for every $\lambda=(\lambda_1,\ldots,\lambda_n)\in\R^n$ we write $f(\lambda)=f(A)$ for any matrix $A\in \mathscr H_n(\mathbb{F}_\beta)$ with eigenvalues $\lambda_1,\ldots,\lambda_n$.
\end{lemma}

\subsection{The Ullman distribution and logarithmic potentials with external fields}\label{subsec:ullman-potential theory}

We call a random variable $\mathbb U$  with values in $[-1,1]$ an Ullman random variable with parameter $p>0$, and write $\mathbb U\sim\mathscr {U}(p)$, if it has density with respect to Lebesgue measure given by
\[
h_p(x):={p\over\pi} \int_{|x|}^1{t^{p-1}\over\sqrt{t^2-x^2}}\,\dint t,\qquad |x|\leq 1.
\]
We notice that $\mathbb U\sim\mathscr{U}(p)$ has the same distribution as the product $AB$, where the random variables $A$ and $B$ are independent and $A$ has an arcsine distribution on $[-1,1]$ with density $x\mapsto{1\over\pi}(1-x^2)^{-1/2}$, while $B$ has a beta distribution with density $x\mapsto px^{p-1}$ with $x\in[0,1]$ (see, e.g., \cite[Lemma~4.1]{VA1987}). In particular, we obtain for $\mathbb U\sim \mathscr U(p)$,
\begin{align}\label{eq:expectation ullman power p}
\E|\mathbb U|^p = \E|A|^p \,\E|B|^p =  \frac{1}{2}\frac{\Gamma(\frac{p+1}{2})}{\sqrt{\pi}\Gamma(\frac{p+2}{2})} .
\end{align}
More generally, the $q$-th absolute moment of $\mathbb U\sim \mathscr U(p)$ is given by
\begin{align}\label{eq:expectation ullman power p_general}
\E|\mathbb U|^q
= \E|A|^q \,\E|B|^q = \frac{p}{p+q}\frac{\Gamma(\frac{q+1}{2})}{\sqrt{\pi}\Gamma(\frac{q+2}{2})}  .
\end{align}
As another consequence of this representation, one derives that the arcsine distribution on $[-1,1]$ is the weak limit of the Ullman distribution, as $p\to\infty$.

The Ullman distribution plays an important r\^ole in the theory of logarithmic potentials with external fields~\cite{SaffBOOK}, in the theory of orthogonal polynomials with respect to Freud weights $e^{-c|x|^p}$ with $c>0$ being a constant \cite{mhaskar_saff,rakhmanov}, and in the theory of random matrices~\cite{hiai_petz,PS2011}. For a probability measure $\mu$ on $\R$ consider the energy functional
\[
\mathscr E_p(\mu) := \int_{\R} \int_{\R} \log \frac{1}{|x-y|}\, \mu(\dint x) \mu(\dint y) +  2 \int_{\R} Q_p(x)\, \mu(\dint x),
\]
where $Q_p(x) := |x|^p/\lambda_p$ defines the `external field' with
\begin{align}\label{eq:lambda p}
\lambda_p := \frac{2}{\sqrt{\pi}} \frac{\Gamma(\frac{p+1}{2})}{\Gamma(\frac{p}{2})} = \frac{p}{\pi}\int_{-1}^1 \frac{|x|^p}{\sqrt{1-x^2}}\, \dint x.
\end{align}
It is known (see, e.g., \cite[Theorem~5.1]{SaffBOOK}) that the Ullman distribution $\mu^{(p)}=\mathscr{U}(p)$ is the unique minimizer of the energy functional $\mathscr E_p$ among all probability measures on $\R$ with finite absolute $p$-th moment.  The above choice of $\lambda_p$ makes the support of the minimizer to be the interval $[-1,1]$.

Let us rephrase some results that shall be used later. The first one is taken from \cite[Lemma 4.3]{mhaskar_saff}.

\begin{lemma}\label{lem:mhaskar-saff}
Let $p>0$. Then, for all $y\in[-1,1]$,
\[
\int_{-1}^1 h_p(x) \log |x-y|\, \dint x = \frac{|y|^p}{\lambda_p} - \log 2 - \frac{1}{p}.
\]
\end{lemma}

The next lemma is a direct consequence of the previous one and determines the `free entropy' of the Ullman distribution with parameter $p>0$.

\begin{lemma}\label{lem:double integral ullman}
Let $p>0$. Then
\[
\int_{-1}^1 \int_{-1}^1 h_p(x)\,h_p(y)\,\log |x-y|\, \dint x\, \dint y = -\log 2 - \frac{1}{2p}.
\]
\end{lemma}
\begin{proof}
We use Lemma \ref{lem:mhaskar-saff} and obtain that
\begin{align*}
\int_{-1}^1 \int_{-1}^1 h_p(x)\,h_p(y)\,\log |x-y|\, \dint x\,\dint y & = \int_{-1}^1  h_p(x)\bigg[ \frac{|y|^p}{\lambda_p} - \log 2 - \frac{1}{p}\bigg] \dint y
 = \frac{\E|\mathbb U|^p}{\lambda_p} - \log 2 - \frac{1}{p}\,,
\end{align*}
where $\mathbb U\sim\mathscr U(p)$. Using \eqref{eq:expectation ullman power p} together with \eqref{eq:lambda p}, we arrive at
\begin{align}\label{eq:p-moment ullman random variable}
\frac{\E|\mathbb U|^p}{\lambda_p} = \frac{1}{2}\frac{\Gamma(\frac{p+1}{2})}{\sqrt{\pi}\Gamma(\frac{p+2}{2})} \cdot \frac{\sqrt{\pi}}{2} \frac{\Gamma(\frac{p}{2})}{\Gamma(\frac{p+1}{2})} = \frac{1}{2p}.
\end{align}
Therefore,
\[
\int_{-1}^1 \int_{-1}^1 h_p(x)\,h_p(y)\,\log |x-y|\, \dint x\,\dint y = -\log 2 - \frac{1}{2p}
\]
and the proof is complete.
\end{proof}

\subsection{Fekete points}\label{subsec:fekete}
Let $E\subset\C$ be an infinite, bounded and closed set. For a positive integer $k\in\N$ the $k$-diameter of $E$ is defined as
\[
\delta_k := \delta_k(E):= \sup_{t_1,\dots,t_k\in E} \bigg(\prod_{1\leq i < j \leq k} |t_j-t_i|\bigg)^{\frac{2}{k(k-1)}}.
\]
The points maximizing the product are called Fekete points. These points are pairwise different and, roughly speaking, maximally spread out over $E$. Note that $\delta_2$ is simply the diameter of the set $E$ while, for $k\geq 3$, $\delta_k$ is the maximum of the geometric means of segments that arise as edges of some complete graph with $n$ nodes in $E$.

It is easily verified that the sequence $(\delta_k)_{k\in\N}$ is non-increasing. Its limit is the so-called transfinite diameter of $E$ and it is well known that the transfinite diameter of a line segment is $1/4$ times its length, see \cite{MF2}. We shall be interested in the particular case where $E=[-1,1]$, whose transfinite diameter equals $1/2$. It is also known that the Fekete points of $[-1,1]$ are the roots of $(1-x^2)P_{n-2}^{(1,1)}(x)$ with $P_{n-2}^{(1,1)}(x)$ being the Jacobi polynomial of order $n-2$ with parameters $(1,1)$ (see, e.g., \cite[p.\ 187]{SaffBOOK}).

For more information on the transfinite diameter and Fekete points we refer the reader to \cite{MF1,MF2,PS1931,SaffBOOK} and references therein.

\subsection{Gauss-Lobatto Chebychev nodes}\label{subsec:gauss-lobatto}

For some $n\in\N$, the Gauss-Lobatto Chebychev nodes are defined by
\[
\widetilde{t}_j := -\cos\Big(\frac{j-1}{n-1}\,\pi\Big)\in\R,\qquad j=1,\dots,n.
\]
This family of points appears frequently in polynomial interpolation as discretization grid to construct an interpolation polynomial. We shall use the following result about the determinant of the Vandermonde matrix at the Gauss-Lobatto Chebychev nodes taken from \cite[Proposition 3]{EF2005}. Let
$$
V_n(x_1,\ldots,x_n) := \begin{pmatrix}
1 & x_1 & x_1^2 & \ldots & x_1^{n-1}\\
1 & x_2 & x_2^2 & \ldots & x_2^{n-1}\\
\vdots & \vdots & \vdots & \ddots & \vdots\\
1 & x_n & x_n^2 & \ldots & x_n^{n-1}\\
\end{pmatrix}
$$
be the Vandermonde matrix based on $x_1,\ldots,x_n\in\R$.

\begin{lemma}\label{lem:vandermonde gauss-lobatto}
Let $n\in\N$ and let $\widetilde{t}_1,\ldots,\widetilde{t}_n$ be the Gauss-Lobatto Chebychev nodes. Then
\[
\det(V_n(\widetilde{t}_1,\ldots,\widetilde{t}_n)) = \prod_{1\leq k<\ell\leq n}(\,\widetilde{t}_\ell - \widetilde{t}_k)
=
2^{n+1-\frac 12 n^2} (n-1)^{\frac{n}{2}}.
\]
\end{lemma}

\begin{rmk} 
Note that Saint Raymond has erroneously missed a factor $2$ in his formula on page 68 in \cite{R1984}.
\end{rmk}

\section{Asymptotic volume of matrix balls}\label{sec:MatrixBallsVolume}

Let us fix $0 < p \leq \infty$, $n\in\N$ and $\beta\in\{1,2,4\}$, and recall that the space $\mathscr H_n(\mathbb{F}_\beta)$ consists of all matrices $A\in \Mat_{n}(\mathbb F_\beta)$ that are self-adjoint, i.e.\ satisfy $A= A^*$. We consider the matrix unit balls
$$
\B_{p,\beta}^n := \Big\{A\in \mathscr H_n(\mathbb{F}_\beta):\sum\limits_{j=1}^n|\lambda_j(A)|^p\leq1\Big\},
$$
which might be regarded as the matrix analogues of the classical $\ell^n_p$-balls. The case $p=\infty$ is interpreted in the usual way.
Using Lemma \ref{lem:IntegrationFormulaForEnsembles} with the appropriate indicator function, we first notice that
$$
\vol_{\beta,n}(\B_{p,\beta}^n) = c_{n,\beta}I_{n,\beta,p}\,,
$$
where the constant $c_{n,\beta}$ is given by \eqref{eq:DefConstantcnbeta} and $I_{n,\beta,p}$ is defined as
\begin{equation}\label{eq:I_n_beta_p}
I_{n,\beta,p} := \int_{\B_p^n}\prod_{1\leq i<j\leq n}|\lambda_j-\lambda_i|^\beta\,\dint \lambda_1 \ldots \dint \lambda_n.
\end{equation}
The eventual goal of this section is to prove the following result about the precise asymptotic volume of unit balls in the matrix ensembles $\mathscr H_n(\mathbb{F}_\beta)$.

\begin{thm}\label{thm:VolumeAsymptotics}
Let $0< p\leq \infty$ and $\beta\in\{1,2,4\}$. Then
$$
\vol_{\beta,n}(\B_{p,\beta}^n)^{2/n^2} \asymptequiv
n^{-\beta({1\over p}+{1\over 2})}\Delta^\beta(p) \left({4\pi\over \beta}\right)^{\beta/2}e^{3\beta/4},
$$
where $\Delta(\infty) = \frac 12$ and, for $p\neq \infty$,
\begin{align*}
\Delta(p)
&= \exp\left\{\int_{-1}^1\int_{-1}^1 h_p(x)h_p(y) \log |x-y| \dint x\,\dint y - \frac{1}{p} \log\E|\mathbb U|^p\right\}=\frac{1}{2}\bigg(\frac{p\sqrt{\pi}\,\Gamma(\frac{p}{2})}{\sqrt{e}\,\Gamma(\frac{p+1}{2})}\bigg)^{\frac{1}{p}}
\end{align*}
with $\mathbb U\sim \mathscr U(p)$ an Ullman random variable.
\end{thm}


The proof of Theorem \ref{thm:VolumeAsymptotics} develops further the ideas from the paper by Saint Raymond \cite{R1984}, who computed the asymptotic volume behaviour of unit balls of Schatten classes without specifying a quantity similar to $\Delta(p)$ for his setting of non self-adjoint matrices. However, the argument needs a careful adaption to our set-up, and for completeness we include all details. We would also like to emphasize that, as the theorem shows, contrary to Saint Raymond we are able to compute the quantity $\Delta(p)$ explicitly for the case of matrix ensembles $\mathscr H_n(\mathbb{F}_\beta)$.  To keep the focus on classical matrix ensembles, we shall present the computation of $\Delta(p)$ in Saint Raymond's setting of non-self-adjoint Schatten $p$-classes in a parallel paper.

\subsection{Asymptotic behaviour of $I_{n,\beta,\infty}$}

We start by determining the asymptotic behaviour of the quantity $I_{n,\beta,\infty}$. Note that its definition in~\eqref{eq:I_n_beta_p} is meaningful for an arbitrary $\beta>0$.

\begin{lemma}\label{lem:asymptotic I infinity}
For any $\beta \in (0,\infty)$, we have
\[
\lim_{n\to\infty}I_{n,\beta,\infty}^{2/n^2} = \frac{1}{2^\beta}.
\]
\end{lemma}
\begin{proof}
For some positive integer $k\in\N$, we consider the $k$-diameter of $[-1,1]$, i.e.,
\[
\delta_k = \sup_{t_1,\dots,t_k\in[-1,1]} \bigg(\prod_{1\leq i < j \leq k} |t_j-t_i|\bigg)^{\frac{2}{k(k-1)}}.
\]
It follows from the discussion in Subsection \ref{subsec:fekete} that, as $k\to\infty$,
\begin{align}\label{eq:convergence delta_k}
\delta_k \downarrow \frac{1}{2},
\end{align}
meaning that $\delta_k$ converges to $1/2$ \textit{from above}, as $k\to\infty$.
For $n\in\N$, we can thus estimate
\begin{align*}
I_{n,\beta,\infty} &  = \int_{\B_\infty^n} \prod_{1\leq i < j\leq n}|\lambda_j-\lambda_i|^\beta \,\dint \lambda_1 \ldots \dint \lambda_n \\
& \leq \int_{\B_\infty^n} \Big(\delta_n^{\frac{n(n-1)}{2}}\Big)^\beta \,\dint \lambda_1 \ldots \dint \lambda_n
 = \vol_{n}(\B_\infty^n)\, \delta_n^{\beta\frac{n(n-1)}{2}}
 = 2^n\, \delta_n^{\beta\frac{n(n-1)}{2}}.
\end{align*}
We conclude from \eqref{eq:convergence delta_k} that
\[
\limsup_{n\to\infty}\, I_{n,\beta,\infty}^{2/n^2} \leq \limsup_{n\to\infty}\, 2^{\frac{2}{n}} \Big(\delta_n^{\frac{n(n-1)}{n^2}}\Big)^{\beta} = \frac{1}{2^\beta}.
\]
We now proceed with the lower bound. To this end, we want to approximate the optimal choice of points in $[-1,1]$ that determine $\delta_n$ by setting
\[
\widetilde{t}_j:= 2\sin^2\bigg(\frac{(j-1)\pi}{2(n-1)}\bigg)-1 = -\cos\bigg(\frac{(j-1)\pi}{n-1}\bigg),\qquad j=1,\dots,n,
\]
i.e., by taking the Gauss-Lobatto Chebychev nodes.
Note that $-1=\widetilde{t}_1 \leq \ldots\leq \widetilde{t}_n=1$ and so this variation (compared to Saint Raymond \cite{R1984}) takes into account that in our case the $t_j$'s come from $[-1,1]$ rather than from $[0,1]$. By Lemma \ref{lem:vandermonde gauss-lobatto}, we have the identity
\begin{align}\label{eq: idendity product cos difference}
\prod_{1\leq i < j \leq n} (\,\widetilde{t}_j-\widetilde{t}_i) = 2^{{n(n-1)\over 2}+1}\bigg(\frac{2(n-1)}{4^{n-1}}\bigg)^{\frac{n}{2}}.
\end{align}
We further observe that 
\begin{align}\label{ineq: m lower bound}
m:= \inf_{i\neq j}|\,\widetilde{t}_i-\widetilde{t}_j|  =  2\sin^2{\pi\over 2(n-1)}\geq \frac{2}{(n-1)^2}.
\end{align}
The equality in the previous display follows from the following consideration. First, since $\cos(\,\cdot\,)$ is strictly decreasing on $[0,\pi/2]$ (which is enough to consider because of the symmetry), the minimum must be attained for two neighboring points. Now, since the modulus of the derivative of $\cos(\,\cdot\,)$ is increasing on $[0,\pi/2]$, it follows that the minimum must be attained for $i=1$ and $j=2$. 

For $\varepsilon\in (0,1/2)$, let us now consider small $\varepsilon m$-neighborhoods of the points $\widetilde{t}_1,\dots,\widetilde{t}_n$. More precisely, we consider the one-sided neighborhoods
\[
t_1\in[\,\widetilde{t}_1,\widetilde{t}_1+\varepsilon m] \quad\text{and}\quad t_n\in[\,\widetilde{t}_n-\varepsilon m,\widetilde{t}_n]
\]
at the two extremal points and for $j\in\{2,\dots,n-1\}$ the two-sided ones
\[
t_j\in [\,\widetilde{t}_j-\varepsilon m, \widetilde{t}_j + \varepsilon m].
\]
Note that these intervals are disjoint. If we write
\[
Q:= [\,\widetilde{t}_1,\widetilde{t}_1+\varepsilon m]\times \left(\prod_{j=2}^{n-1} [\,\widetilde{t}_j-\varepsilon m, \widetilde{t}_j + \varepsilon m]\right) \times [\,\widetilde{t}_n-\varepsilon m,\widetilde{t}_n] ,
\]
then
\begin{align}\label{eq: volume Q}
\vol_{n}(Q)= (\varepsilon m)^2\cdot(2\varepsilon m)^{n-2} = 2^{n-2}(\varepsilon m)^n.
\end{align}
Let us define
\[
\varphi_\beta:[-1,1]^n\to\R,\quad \varphi_\beta(t_1,\dots,t_n):= \prod_{1\leq i < j\leq n}|t_j-t_i|^\beta.
\]
The idea is to estimate the integral over $\varphi_\beta$ on $\B_\infty^n$ from below essentially by the values of $\varphi_\beta$ at the points $(t_1,\dots,t_n)$ in the $n$-dimensional box $Q$. To this end, we first observe that for $(t_1,\dots,t_n)\in Q$,
\begin{align}\label{ineq: inequality t_i-t_j}
|t_i-t_j| \geq |\,\widetilde{t}_i-\widetilde{t}_j|-2\varepsilon m \geq |\,\widetilde{t}_i - \widetilde{t}_j|(1-2\varepsilon),
\quad i\neq j,
\end{align}
where we have used that $m\leq |\widetilde{t}_i-\widetilde{t}_j|$. Therefore, using \eqref{ineq: inequality t_i-t_j} we obtain for all $(t_1,\ldots,t_n)\in Q$ the estimate
\begin{align}\label{ineq: phi on box Q}
\varphi_\beta(t_1,\dots,t_n) \geq (1-2\varepsilon)^{\beta\frac{n(n-1)}{2}} \varphi_\beta\big(\,\widetilde{t}_1,\dots,\widetilde{t}_n\big).
\end{align}
Putting things together, we obtain from \eqref{ineq: phi on box Q} that
\begin{align*}
I_{n,\beta,\infty}^{2/n^2} & = \bigg(\int_{\B_\infty^n} \varphi_\beta(t_1,\dots,t_n)\,\dint t_1 \ldots \dint t_n\bigg)^{\frac{2}{n^2}} \cr
& \geq \bigg(\int_{Q} \varphi_\beta(t_1,\dots,t_n)\,\dint t_1 \ldots \dint t_n\bigg)^{\frac{2}{n^2}} \cr
& \geq \vol_n(Q)^{\frac{2}{n^2}}\cdot \Big[(1-2\varepsilon)^{\beta\frac{n(n-1)}{2}} \varphi_\beta\big(\,\widetilde{t}_1,\dots,\widetilde{t}_n\big)\Big]^{\frac{2}{n^2}}.
\end{align*}
Choosing $\varepsilon:=\frac{1}{2(n+1)}$ and using \eqref{ineq: m lower bound} together with \eqref{eq: volume Q}, we obtain
\begin{align*}
I_{n,\beta,\infty}^{2/n^2} &\geq \Big[2^{n-2}(\varepsilon m)^n\Big]^{\frac{2}{n^2}}\cdot \bigg[\Big(1-\frac{1}{n+1}\Big)^{\beta\frac{n(n-1)}{2}} \varphi_\beta\big(\,\widetilde{t}_1,\dots,\widetilde{t}_n\big)\bigg]^{\frac{2}{n^2}} \cr
& = \bigg[2^{n-2}\bigg(\frac{1}{(n+1)(n-1)^2}\bigg)^n\bigg]^{\frac{2}{n^2}}\cdot \Big(1-\frac{1}{n+1}\Big)^{\beta\frac{n-1}{n}} \varphi_\beta\big(\,\widetilde{t}_1,\dots,\widetilde{t}_n\big)^{\frac{2}{n^2}}.
\end{align*}
We now observe that, using the identity \eqref{eq: idendity product cos difference}, we have
\[
\varphi_\beta\big(\,\widetilde{t}_1,\dots,\widetilde{t}_n\big)^{\frac{2}{n^2}}
=
\bigg[2^{n + 1 - \frac 12 n^2} (n-1)^{\frac{n}{2}}\bigg]^{\beta \frac{2}{n^2}}
=
2^{\frac{\beta}{n^2} (-n^2 + 2n + 2)} (n-1)^{\frac{\beta}{n}}.
\]
Altogether, we obtain
\[
\liminf_{n\to\infty} I_{n,\beta,\infty}^{2/n^2} \geq \frac{1}{2^\beta},
\]
which completes the proof.
\end{proof}

\subsection{Asymptotic behavior of the constants $c_{n,\beta}$}

Next, we consider the asymptotic behavior of the constants $c_{n,\beta}$ defined in \eqref{eq:DefConstantcnbeta}, as $n\to\infty$. Again, we treat the case of a general parameter $\beta\in(0,\infty)$.

\begin{lemma}\label{lem:AsymptoticConstant}
For any $\beta\in (0,\infty)$, we have
\[
c_{n,\beta}^{2/n^2}\asymptequiv n^{-\frac{\beta}{2}}\Big({4\pi\over \beta}\Big)^{\frac{\beta}{2}}e^{\frac{3\beta}{4}}\,.
\]
\end{lemma}
\begin{proof}
Using that $\log\Gamma(z)=z\log z-z+o(z)$ and $\log n!= o(n^2)$, as $z,n\to\infty$, we obtain
\begin{align*}
\log c_{n,\beta} &= \sum_{k=1}^n\Big(\log 2+{\beta k\over 2}\log(2\pi)-{\beta\over 2}\log 2-\log\Gamma\Big({\beta k\over 2}\Big)\Big) - \log n! - n\log{2\pi^{\beta/2}\over\Gamma\big({\beta\over 2}\big)}\\
&=\sum_{k=1}^n\Big({\beta k\over 2}\log(2\pi)- {\beta k\over 2}\log{\beta k\over 2}+{\beta k\over 2}+o(k)\Big) + o(n^2)\\
&=-\sum_{k=1}^n{\beta k\over 2}\log{\beta k\over 2}+{n^2\beta\over 4}\log(2\pi)+{n^2\beta\over 4} +o(n^2).
\end{align*}
Next we determine the large $n$ behaviour of the sum $\sum_{k=1}^n{\beta k\over 2}\log{\beta k\over 2}$. Using Abel's partial summation formula
\[
\sum_{k=1}^na_kb_k=A_nb_n-\sum_{k=1}^{n-1}A_k(b_{k+1}-b_k),\qquad A_k:=a_1+\ldots+a_k,
\]
with the choices $a_k=\beta k/2$ and $b_k=\log(\beta k/2)$, we see that
\begin{align*}
\sum_{k=1}^n{\beta k\over 2}\log{\beta k\over 2} &= {\beta n(n+1)\over 4}\log{\beta n\over 2}-\sum_{k=1}^{n-1}{\beta k(k+1)\over 4}\Big(\log{\beta(k+1)\over 2}-\log{\beta k\over 2}\Big)\\
&={\beta n^2\over 4}\log{\beta n\over 2}+{\beta n\over 4}\log{\beta n\over 2}-\sum_{k=1}^{n-1}{\beta(k+1)\over 4}+\sum_{k=1}^{n-1}{\beta k(k+1)\over 4}\Big({1\over k}-\log\Big(1+{1\over k}\Big)\Big)\\
&={\beta n^2\over 4}\log{\beta n\over 2}-{\beta n^2\over 8}+O(n\log n),
\end{align*}
since $\sum_{k=1}^{n-1}{\beta(k+1)\over 4}={\beta\over 8}(n^2+n-2)$ and $\frac 1k - \log (1+\frac 1k) =O(\frac 1 {k^2})$. Inserting this into the above expression for $\log c_{n,\beta}$, we arrive at
\begin{align*}
\log c_{n,\beta} &= -{\beta n^2\over 4}\log{\beta n\over 2}+n^2\Big({\beta\over 4}\log(2\pi)+{\beta\over 4}+{\beta\over 8}\Big)+o(n^2)\\
&=-{\beta n^2\over 4}\log n+n^2\Big({3\beta\over 8}+{\beta\over 4}\Big(\log(2\pi)-\log{\beta\over 2}\Big)\Big)+o(n^2)\\
&=-{\beta n^2\over 4}\log n+n^2\Big({3\beta\over 8}+{\beta\over 4}\log{4\pi\over\beta}\Big)+o(n^2).
\end{align*}
Thus,
\begin{align*}
\log c_{n,\beta}^{2/n^2} &= {2\over n^2}\Big(-{\beta n^2\over 4}\log n+n^2\Big({3\beta\over 8}+{\beta\over 4}\log{4\pi\over\beta}\Big)+o(n^2)\Big)\\
&=-{\beta\over 2}\log n+{3\beta\over 4}+\log\Big({4\pi\over\beta}\Big)^{\beta/2}+o(1)
\end{align*}
and we conclude that
\begin{align*}
c_{n,\beta}^{2/n^2} &\asymptequiv \exp\Big(-\frac \beta 2 \log n + {3\beta\over 4}+\log\Big({4\pi\over\beta}\Big)^{\beta/2}\Big)
=n^{-\beta/2}\Big({4\pi\over\beta}\Big)^{\beta/2}e^{3\beta/4}.
\end{align*}
This proves the claim.
\end{proof}

\subsection{The parameter $\Delta_n(p)$ and its asymptotic behavior}

If not specified otherwise, in this subsection we shall always assume that $0 < p <\infty$ . Let us define the quantity
\begin{align}\label{eq:def delta_n(p)}
\Delta_n(p) := \sup_{(t_1,\dots,t_n)\in\R^n\setminus\{0\}}\,n^{1/p}\frac{\bigg(\prod\limits_{1\leq i < j\leq n}|t_i-t_j|\bigg)^{\frac{2}{n(n-1)}}}{\Big(\sum\limits_{i=1}^n|t_i|^p\Big)^{\frac{1}{p}}}.
\end{align}
We first prove that a  maximizer for $\Delta_n(p)$ exists in the unit $\ell_p^n$-sphere and that the gaps between the consecutive $t_i$'s are not `too small'. Saint Raymond~\cite{R1984} has a similar result for the supremum taken over positive $t_i$'s, but in our two-sided setting additional difficulties appear. In particular, we are not able to estimate the size of the gap containing $0$ and shall treat it separately.

\begin{lemma}\label{lem: delta_n(p)}
For $n\geq 3$ there exist $t_{1,n}^*,\dots,t_{n,n}^*\in\R$ with $t_{1,n}^*< \ldots < t_{n,n}^*$, $\|(t_{1,n}^*,\dots,t_{n,n}^*) \|_p=1$, and
\[
\prod_{1\leq i < j \leq n}|t_{i,n}^*-t_{j,n}^*| = \bigg(\frac{\Delta_n(p)}{n^{\frac{1}{p}}}\bigg)^{\frac{n(n-1)}{2}}
\]
such that
\[
m_1 := \inf_{k_0 \leq i < j\leq n} |t_{j,n}^*-t_{i,n}^*| \geq n^{-2-\frac{2}{p}}\qquad\text{and}\qquad m_2 := \inf_{1 \leq i < j\leq k_0-1} |t_{j,n}^*-t_{i,n}^*| \geq n^{-2-\frac{2}{p}},
\]
where $k_0\in\{2,\dots,n\}$ is chosen in such a way that $t_{k_0,n}^*\geq 0$ and $t_{k_0-1,n}^*<0$.
\end{lemma}
\begin{proof}
It follows from compactness and the homogeneity that there exist $t_{1,n}^*,\dots,t_{n,n}^*\in\R$ such that
\[
\|(t_{1,n}^*,\dots,t_{n,n}^*) \|_p=1 \qquad\text{and}\qquad \prod_{1\leq i < j \leq n}|t_{i,n}^*-t_{j,n}^*| = \bigg(\frac{\Delta_n(p)}{n^{\frac{1}{p}}}\bigg)^{\frac{n(n-1)}{2}}.
\]
To simplify the notation, we shall write $t_i^*$ instead of $t_{i,n}^*$ until the end of this proof.
By permuting the entries, we can assume that $t_1^*\leq \ldots \leq t_n^*$. Since the product is zero if $t_i^*=t_j^*$, we may even assume that $t_1^*< \ldots < t_n^*$. Let us show that there are both positive and negative terms among the $t_{i}^*$'s. Indeed, if all terms were strictly positive, then we could consider the $n$-tuple $(0, t_{2}^*-t_{1}^*,\ldots,t_{n}^*-t_{1}^*)$, which would lead to the same Vandermonde determinant but also to a strictly smaller $\ell^p$-norm, a contradiction. If all terms were strictly positive except for $t_1^*=0$, then we could consider the $n$-tuple $(-t_n^*, t_1^*,\ldots,t_{n-1}^*)$ which has the same $\ell^p$-norm but a larger Vandermonde determinant than the original $n$-tuple, again a contradiction. So, there are strictly negative elements among the $t_i^*$'s. Similarly, there are strictly positive ones.

It is left to prove the estimates for $m_1$ and $m_2$, which is the most difficult part of the proof.  Due to symmetry, it suffices to bound $m_1$ from below and to note that the same estimate carries over to $m_2$.
We observe that for pairwise distinct $t_1,\dots,t_n\in\R\setminus\{0\}$ and $k\in\{1,\dots,n\}$,
\[
\frac{\partial}{\partial t_k}\sum_{1\leq i < j \leq n}\log|t_i-t_j| = \sum_{i=1\atop i\neq k}^n \frac{1}{t_k-t_i}.
\]
On the other hand,
\[
\frac{\partial}{\partial t_k} \sum_{i=1}^n|t_i|^p = p|t_k|^{p-1}\sgn(t_k),
\]
where for $x\in\R$, $\sgn(x)=-1$ if $x<0$, $\sgn(x)=0$ if $x=0$ and $\sgn(x)=+1$ provided that $x>0$.
We therefore obtain the Lagrange conditions
\[
\alpha_k:=\sum_{i=1\atop i\neq k}^n \frac{1}{t_k^*-t_i^*} = \lambda |t_k^*|^{p-1}\sgn(t_k^*),\qquad k=1,\dots,n,
\]
with some parameter $\lambda\in\R$.
Note that even though the function above is not differentiable at $0$ for $p\leq 1$, this will cause no difficulties  in what follows.  In fact, if $t^*_{k_0}=0$, then we always use the respective Lagrange condition in the form $\alpha_{k_0}t_{k_0}^*=0$, which is trivially fulfilled.
Let us now determine the parameter $\lambda$. To do this, we compute the sum over $k$ of $\alpha_kt_k^*$ in two different ways. First,
\begin{align}\label{eq:computation sum alpha_kt_k - 1}
\sum_{k=1}^n\alpha_kt_k^* = \sum_{k=1}^n \sum_{i=1\atop i\neq k}^n\frac{t_k^*}{t_k^*-t_i^*}
= \sum_{(i,k)\atop i\neq k} \frac{t_k^*}{t_k^*-t_i^*} = \frac{n(n-1)}{2},
\end{align}
where the last equality simply follows because for the pairs $(i,k)$ we obtain $\frac{t_k^*}{t_k^*-t_i^*}$ and for the pairs $(k,i)$ we get $\frac{t_i^*}{t_i^*-t_k^*}$, for which the sum is then equal to $1$. On the other hand, since $\|(t_k^*)_{k=1}^n\|_p=1$,
\begin{align}\label{eq:computation sum alpha_kt_k - 2}
\sum_{k=1}^n\alpha_kt_k^* = \sum_{k=1}^n \lambda |t_k^*|^{p-1}\sgn(t_k^*)t_k^* = \lambda \sum_{k=1}^n |t_k^*|^p = \lambda.
\end{align}
Therefore, from \eqref{eq:computation sum alpha_kt_k - 1} and \eqref{eq:computation sum alpha_kt_k - 2}, we deduce that
\[
\lambda = \frac{n(n-1)}{2}.
\]
We now prove upper and lower bounds for the $k$-truncated sum of the $\alpha_i$'s. We have, for any $k>k_0$,
\begin{align}\label{eq:upper bound alphi from k}
\sum_{i=k}^n \alpha_i = \sum_{i=k}^n\lambda |t_i^*|^{p-1}\sgn(t_i^*) \leq \frac{n(n-1)}{2}\sum_{i=k_0+1}^n|t_i^*|^{p-1}.
\end{align}
For $k\in\{2,\dots,n\}$, we obtain the lower bound
\begin{align}\label{eq:lower bound alphi from k}
\sum_{i=k}^n\alpha_i & = \sum_{i=k}^n \sum_{j=1\atop j\neq i}^n \frac{1}{t_i^*-t_j^*}
 = \sum_{(i,j)\atop j<k\leq i}\frac{1}{t_i^*-t_j^*} + \sum_{(i,j):\,i\neq j\atop k\leq i,\,k\leq j}\frac{1}{t_i^*-t_j^*} \cr
& =  \sum_{(i,j)\atop j<k\leq i}\frac{1}{t_i^*-t_j^*} \geq \sum_{i=k}^n \frac{1}{t_i^*-t_{k-1}^*},
\end{align}
and whenever $k>k_0$, we get
\begin{align}\label{eq:lower bound alphi from k>k_0}
\sum_{i=k}^n\alpha_i & = \sum_{(i,j):\,j<k\leq i}\frac{1}{t_i^*-t_j^*} \geq \sum_{i=k}^n\frac{1}{t_i^*-t_{k_0}^*}.
\end{align}

Before we continue with our estimates, observe that by H\"older's inequality and the fact that $\|(t_1^*,\dots,t_n^*)\|_p=1$,
\begin{align}\label{ineq: hoelder estimate}
\sum_{i=k_0+1}^n |t_i^*|^{p-1} & = \sum_{i=k_0+1} |t_i^*|^{p\frac{p}{p+1}}\bigg(\frac{1}{t_i^*}\bigg)^{\frac{1}{p+1}} \cr
 & \leq \bigg(\sum_{i=k_0+1}^n|t_i^*|^p\bigg)^{\frac{p}{p+1}}\bigg(\sum_{i=k_0+1}^n
\frac{1}{t_i^*}\bigg)^{\frac{1}{p+1}}
 \leq  \bigg(\sum_{i=k_0+1}^n
\frac{1}{t_i^*}\bigg)^{\frac{1}{p+1}}.
\end{align}
We immediately obtain from \eqref{eq:lower bound alphi from k>k_0} together with \eqref{eq:upper bound alphi from k} and \eqref{ineq: hoelder estimate} that, for any $k>k_0$,
\begin{align*}
\sum_{i=k}^n \frac{1}{|t_i^*|}  & = \sum_{i=k}^n \frac{1}{t_i^*}
 \leq \sum_{i=k}^n \frac{1}{t_i^*-t_{k_0}^*}
\leq \frac{n(n-1)}{2} \sum_{i=k_0+1}^n|t_i^*|^{p-1} \cr  &\leq \frac{n(n-1)}{2} \bigg(\sum_{i=k_0+1}^n
\frac{1}{t_i^*}\bigg)^{\frac{1}{p+1}}.
\end{align*}
Using the latter with $k=k_0+1$ and rearranging the resulting inequality, we find that
\begin{align}\label{ineq: sum 1 over t_i^*}
\sum_{i=k_0+1}^n
\frac{1}{t_i^*} \leq n^{\frac{2(p+1)}{p}}\,.
\end{align}
Similarly, using again H\"older's inequality as before and \eqref{ineq: sum 1 over t_i^*}, we obtain
\begin{align}\label{ineq: sum t_i^* with power p-1 upper bound}
\sum_{i=k_0+1}^n |t_i^*|^{p-1} \leq \bigg(\sum_{i=k_0+1}^n
\frac{1}{t_i^*}\bigg)^{\frac{1}{p+1}} \leq n^{\frac{2}{p}}\,.
\end{align}
We are now in the position to estimate from below the parameter $m_1$ (defined in the statement of the lemma and determining the minimal gap between the non-negative elements of the sequence of maximizers). For $k>k_0$ (and $n\geq 3$), using \eqref{eq:lower bound alphi from k} together with \eqref{eq:upper bound alphi from k} and \eqref{ineq: sum t_i^* with power p-1 upper bound},
\begin{align*}
\frac{1}{t^*_k-t^*_{k-1}} & \leq \sum_{i=k}^{n}\frac{1}{t_i^*-t_{k-1}^*} \leq \sum_{i=k}^n \alpha_i \cr & \leq \frac{n(n-1)}{2}\sum_{i=k_0+1}^n |t_i^*|^{p-1}
\leq \frac{n(n-1)}{2}n^{\frac{2}{p}} \leq n^{2+\frac{2}{p}}\,.
\end{align*}
Therefore, for all $k>k_0$, we get $t_k^*-t_{k-1}^* \geq n^{-2-\frac{2}{p}}$ and so
\[
m_1 \geq n^{-2-\frac{2}{p}}\,.
\]
The proof is thus complete.
\end{proof}

Let us now prove the monotonicity of the sequence $(\Delta_n(p))_{n\in\N}$, which at the same time implies its convergence. The value of this limit will be computed in Theorem \ref{lem:limit delta_n(p)} below.

\begin{lemma}\label{lem: convergence delta_n(p)}
As $n\to\infty$, we have the monotone convergence $\Delta_n(p)\,\downarrow\, \Delta(p):= \inf_{n}\Delta_n(p)$.
\end{lemma}
\begin{proof}
We prove that $\Delta_n(p)$ is decreasing in $n$.
Let $n\in\N$ and $(t_1,\ldots,t_{n+1}) \in \R^{n+1} \backslash \{0\}$.  Again, as in Lemma \ref{lem: delta_n(p)}, we may assume $t_1\leq t_2 \leq \ldots \leq t_{n+1}$.  We obtain from the arithmetic-geometric-mean inequality that
\begin{align*}
\prod_{1\leq i < j \leq n+1} |t_i-t_j|^{n-1} & = \prod_{k=1}^{n+1}\bigg(\prod_{1\leq i < j \leq n+1\atop{i,j\neq k}}|t_j-t_i|\bigg) \cr
& \leq \prod_{k=1}^{n+1}\left(\bigg(\sum_{i=1\atop i\neq k}^n|t_i|^p\bigg)^{\frac{1}{p}} \Delta_n(p)\,n^{-\frac{1}{p}}\right)^{\frac{n(n-1)}{2}} \cr
& = \bigg(\frac{\Delta_n(p)}{n^{\frac{1}{p}}}\bigg)^{\frac{n(n+1)(n-1)}{2}}\left(\prod_{k=1}^{n+1}\bigg(\sum_{i=1\atop i\neq k}^n|t_i|^p\bigg) \right)^{\frac{n(n-1)}{2p}} \cr
& \leq \bigg(\frac{\Delta_n(p)}{n^{\frac{1}{p}}}\bigg)^{\frac{n(n+1)(n-1)}{2}}\left(\frac{1}{n+1}\sum_{k=1}^{n+1}\sum_{i=1\atop i\neq k}^n|t_i|^p\right)^{\frac{n(n+1)(n-1)}{2p}} \cr
& \leq \bigg(\frac{\Delta_n(p)}{n^{\frac{1}{p}}}\bigg)^{\frac{n(n+1)(n-1)}{2}}\left(\frac{n}{n+1}\sum_{k=1}^{n+1}|t_k|^p\right)^{\frac{n(n+1)(n-1)}{2p}} \cr
& = \bigg(\Delta^p_n(p)\,\frac{1}{n+1} \sum_{k=1}^{n+1}|t_k|^p\bigg)^{\frac{n(n+1)(n-1)}{2p}}.
\end{align*}
Therefore, by rearranging the latter estimate, we obtain
\[
(n+1)^{1/p}\frac{\bigg(\prod\limits_{1\leq i < j\leq n+1}|t_i-t_j|\bigg)^{\frac{2}{n(n+1)}}}{\Big(\sum\limits_{k=1}^{n+1}|t_k|^p\Big)^{\frac{1}{p}}} \leq \Delta_n(p),
\]
which immediately implies $\Delta_{n+1}(p) \leq \Delta_n(p)$. In particular,  $\Delta_n(p)$, $n\geq 2$ converges to $\Delta(p):=\inf_{n}\Delta_n(p)$ from above.

\end{proof}

We shall now determine the limit of the sequence $(\Delta_n(p))_{n\in\N}$, which we denote (as already done in the previous lemma) by $\Delta(p)$. The proof uses results from the theory of logarithmic potentials with external fields, in particular, the extremal properties of the Ullman distribution. We start with a preliminary lemma that is required in our proof of the upper bound. Let us recall that $h_p$ denotes the density of the Ullman distribution $\mathscr U(p)$ with parameter $p$.

\begin{lemma}\label{lem:J_p_maximizer}
Let $p>0$. On the set of probability measures $\mu$ on $\R$ with $\int_{\R} |x|^p \mu(\dint x) <\infty$, excluding the Dirac measure at $0$, we consider the functional
\[
\mathscr J_p(\mu):= \int_{\R} \int_{\R} \log|x-y|\,\mu(\dint x)\,\mu(\dint y) - \frac{1}{p}\log \int_{\R} |x|^p \,\mu(\dint x) \in \R\cup\{-\infty\}.
\]
Then the only maximizers of $\mathscr J_p$ are probability measures $\mu_b^{(p)}$ with densities $\frac{1}{b}h_p(\frac{x}{b})$, $b>0$.
\end{lemma}
\begin{proof}
First, we observe that
\begin{align}\label{eq: J_p in terms of expectation}
\mathscr J_p(\mu) = \E \log|Z-\widetilde Z| - \frac{1}{p}\log \E|Z|^p,
\end{align}
where $Z$ is a random variable with distribution $\mu$ and $\widetilde Z$ is an independent copy of $Z$. In particular, the expression $\mathscr J_p(\mu)$ is invariant under scaling by constants $c\in(0,\infty)$, that is, we may replace $Z$ by $cZ$ without changing the value of the right-hand side in \eqref{eq: J_p in terms of expectation}. Therefore, we can scale in such a way that the $p$-th moment equals the one of an Ullman random variable with parameter $p>0$, i.e.,
\[
\int_{\R} |x|^p \,\mu(\dint x) = \int_{\R}|x|^p h_p(x) \,\dint x = \frac{\lambda_p}{2p}.
\]
The last equality follows from Equation \eqref{eq:p-moment ullman random variable}. Consequently, it is enough to show that among all probability measures $\mu$ on $\R$ with $p$-th absolute moment equal to $\frac{\lambda_p}{2p}$, the Ullman distribution $\mu^{(p)}$ is the only maximizer of the expression
\[
\int_{\R} \int_{\R} \log|x-y|\,\mu(\dint x)\,\mu(\dint y).
\]
But this fact is known, see~\cite[Proposition 5.3.4]{hiai_petz}. Alternatively, one can argue as follows.
Assume that $\mu$ has the required $p$-th absolute moment but
\[
\int_{\R} \int_{\R} \log|x-y|\,\mu(\dint x)\,\mu(\dint y) \geq \int_{\R} \int_{\R} h_p(x) h_p(y )\log|x-y|\,\dint x\,\dint y.
\]
Then
\begin{multline*}
\mathscr E_p(\mu) = \int_{\R} \int_{\R} \log \frac{1}{|x-y|} \,\mu(\dint x)\,\mu(\dint y) + \frac{2}{\lambda_p} \int_{\R}|x|^p \,\mu(\dint x)
\\
\leq \int_{\R} \int_{\R} h_p(x) h_p(y)\log \frac{1}{|x-y|} \,\dint x\,\dint y + \frac{2}{\lambda_p} \int_{\R}|x|^ph_p(x) \,\dint x = \mathscr E_p(\mu^{(p)}).
\end{multline*}
The unique minimizer of $\mathscr E_p$ is the Ullman distribution $\mu^{(p)}$, see~\cite[Theorem 5.1 on p.~240]{SaffBOOK}. It follows that $\mu = \mu^{(p)}$. Noting that the value of $\mathscr J_p(\,\cdot\,)$ remains unchanged if we replace $\mu^{(p)}$ by a measure $\mu_b^{(p)}$, $b\in(0,\infty)$, as in the statement of the lemma, completes the proof.
\end{proof}

We now present the main result of this subsection.

\begin{thm}\label{lem:limit delta_n(p)}
As $n\to\infty$,
\[
\log \Delta_n(p) \rightarrow \log \Delta(p) = \int_{-1}^1\int_{-1}^1 h_p(x)h_p(y) \log |x-y| \dint x\,\dint y - \frac{1}{p} \log\E|\mathbb U|^p,
\]
where $\mathbb U\sim \mathscr U(p)$ is an Ullman random variable.
\end{thm}
\begin{rmk}
Recalling Lemma~\ref{lem:double integral ullman} and~\eqref{eq:expectation ullman power p}, we obtain the explicit value of $\Delta(p)$ given in Theorem~\ref{thm:VolumeAsymptotics}.
\end{rmk}
\begin{proof}[Proof of Theorem~\ref{lem:limit delta_n(p)}]
Recall the definition of $\Delta_n(p)$ in Equation \eqref{eq:def delta_n(p)}. We have that
\begin{equation}\label{eq:log_Delta_n}
\log \Delta_n(p) = \sup_{(t_1,\ldots,t_n)\in\R^n\backslash\{0\}} \left(\frac{2}{n(n-1)} \sum_{1\leq i < j \leq n} \log |t_i-t_j| - \frac 1p \log \left(\frac 1n \sum_{i=1}^n |t_i|^p\right)\right).
\end{equation}
Let us start with the lower bound and consider independent random variables $t_1,t_2,\ldots \sim\mathscr U(p)$ with density $h_p$. The law of large numbers for $U$-statistics~\cite[Theorem 3.1.1]{koroljuk_book} implies that, as $n\to\infty$,
\[
\frac{2}{n(n-1)} \sum_{1\leq i < j \leq n} \log|t_i-t_j| \stackrel{\text{a.s.}}{\longrightarrow} \int_{-1}^1\int_{-1}^1 h_p(x)h_p(y) \log|x-y|\, \dint x\,\dint y.
\]
On the other hand, the classical strong law of large numbers shows that, as $n\to\infty$,
\begin{align*}
\frac{1}{p}\log\bigg(\frac{1}{n}\sum_{i=1}^n|t_i|^p\bigg) \stackrel{\text{a.s.}}{\longrightarrow} \frac{1}{p} \log \E|\mathbb U|^p,
\end{align*}
where $\mathbb U\sim \mathscr U(p)$. Therefore, as $n\to\infty$,
\begin{multline*}
\frac{2}{n(n-1)} \sum_{1\leq i < j \leq n} \log |t_i-t_j| - \frac 1p \log \left(\frac 1n \sum_{i=1}^n |t_i|^p\right)
\\
\stackrel{\text{a.s.}}{\longrightarrow} \int_{-1}^1\int_{-1}^1 h_p(x)h_p(y) \log|x-y|\, \dint x\,\dint y - \frac{1}{p} \log \E|\mathbb U|^p.
\end{multline*}
Hence,
\[
\log \Delta(p) = \lim_{n\to\infty}\log \Delta_n(p) \geq \int_{-1}^1\int_{-1}^1 h_p(x)h_p(y) \log |x-y|\, \dint x\,\dint y - \frac{1}{p} \log\E|\mathbb U|^p.
\]

\vskip 2mm
We continue with the upper bound. To this end, we consider a maximizer of the right-hand side of~\eqref{eq:log_Delta_n}, which we denote by $(t_{1,n}^*,\dots,t_{n,n}^*)$. As before, we may assume that
\[
t_{1,n}^* < \ldots < t_{n,n}^* \qquad\text{and}\qquad \|(t_{1,n}^*,\ldots,t_{n,n}^*)\|_p=1,
\]
and that, as we demonstrated in Lemma \ref{lem: delta_n(p)},
\begin{equation}\label{eq:gaps_big}
\inf_{k_0 \leq i < j\leq n} |t_{j,n}^*-t_{i,n}^*| \geq n^{-c_p}\quad\text{and}\quad \inf_{1 \leq i < j\leq k_0-1} |t_{j,n}^*-t_{i,n}^*| \geq  n^{-c_p},
\end{equation}
where $c_p:= 2+\frac 2p$ and $k_0=k_0(n)$ is such that $t^*_{k_0-1,n} < 0 \leq t^*_{k_0,n}$.
This essentially means that, for all $1\leq i < j\leq n$,
\[
|t_{j,n}^*-t_{i,n}^*| \geq n^{-2-\frac{2}{p}},
\]
but excludes the case where $i=k_0-1$ and $j=k_0$. Let us put $\eps_n := n^{-2 c_p}$ and consider the following (absolutely continuous) probability measure $\nu_n$ on $\R$, which is the uniform measure on appropriate one-sided neighborhoods $B_{i,n}$ of the maximizing points $t_{i,n}^*$ with density
$$
f_n(t) =
\frac 1 {n\eps_n}
\sum_{i=1}^n \mathbbm 1_{B_{i,n}}(t),\qquad t\in \R,
$$
where
$$
B_{i,n}
=
\begin{cases}
[t_{i,n}^*-\eps_n, t^*_{i,n}] &:\, i>k_0\\
[t_{i,n}^*, t^*_{i,n}+\eps_n] &:\, i<k_0-1\\
[t_{k_0,n}^*,t_{k_0,n}^*+\eps_n] &:\, i=k_0\\
[t_{k_0-1,n}^*-\eps_n,t_{k_0-1,n}^*] &:\, i=k_0-1.
\end{cases}
$$
For sufficiently large $n$, the intervals $B_{1,n},\ldots,B_{n,n}$ are disjoint by~\eqref{eq:gaps_big} and $f_n$ is indeed a probability density because $\int_\R f_n(t) \dint t = 1$. We claim that
\begin{equation}\label{eq:limsup_p_moment_f_n}
\limsup_{n\to\infty} n\int_\R |t|^p f_n(t) \dint t \leq 1
\end{equation}
and
\begin{equation}\label{eq:log_energy_f_n}
\int_\R\int_\R f_n(x) f_n(y) \log |x-y| \,\dint x\, \dint y
\geq \frac{2}{n^2} \sum_{1\leq i < j \leq n} \log |t_{i,n}^* - t_{j,n}^*| -o(1),
\end{equation}
where $o(1)$ stands for a sequence which tends to $0$, as $n\to\infty$.
Noting that \eqref{eq:limsup_p_moment_f_n} implies
\[
-\limsup_{n\to\infty} \frac{1}{p} \log \left( n\int_\R |t|^p f_n(t)\dint t\right) \geq 0,
\]
from these two claims it would follow that the functional $\mathscr J_p(\,\cdot\,)$ defined as in Lemma~\ref{lem:J_p_maximizer} satisfies
\begin{eqnarray*}
&&\liminf_{n\to\infty} \mathscr J_p(\nu_n)\\
&=&
\liminf_{n\to\infty} \left(\int_\R\int_\R f_n(x) f_n(y) \log |x-y| \, \dint x\, \dint y - \frac 1p \log \bigg( n\int_\R |t|^p f_n(t)\,\dint t\bigg) - \frac 1p \log \frac 1n\right)
\\
&\geq&
\liminf_{n\to\infty}
\left(\frac{2}{n^2} \sum_{1\leq i < j \leq n} \log |t_{i,n}^* - t_{j,n}^*| -o(1) - \frac{1}{p}\log \frac{1}{n}\right)\\
&=&
\liminf_{n\to\infty}
\left[\bigg(\frac{2}{n(n-1)} \sum_{1\leq i < j \leq n} \log |t_{i,n}^* - t_{j,n}^*| - \frac{n^2}{n(n-1)}\frac{1}{p}\log \frac{1}{n}\bigg)\cdot \frac{n(n-1)}{n^2}\right] \\
&\geq&
\liminf_{n\to\infty}
\left[\bigg(\frac{2}{n(n-1)} \sum_{1\leq i < j \leq n} \log |t_{i,n}^* - t_{j,n}^*| - \frac{1}{p}\log \frac{1}{n}\bigg)\cdot \frac{n(n-1)}{n^2}\right] \\
&=& \liminf_{n\to\infty} \log\Delta_n(p) = \log\Delta(p),
\end{eqnarray*}
where we used \eqref{eq:limsup_p_moment_f_n} and \eqref{eq:log_energy_f_n} for the first inequality and Lemma~\ref{lem: convergence delta_n(p)} in the last step.
Together with Lemma~\ref{lem:J_p_maximizer}, this yields that $\log \Delta(p) \leq \liminf_{n\to\infty} \mathscr J_p(\nu_n) \leq \mathscr J_p(\mu^{(p)})$, where $\mu^{(p)}$ is the Ullman measure, and the upper bound would follow.

\vspace*{2mm}
\noindent
\textit{Proof of~\eqref{eq:limsup_p_moment_f_n}.} We have
\[
n\int_0^{\infty} |t|^p f_n(t) \,\dint t = \frac 1 {\eps_n} \sum_{i=k_0}^n  \int_{B_{i,n}} |t|^p \,\dint t
\leq(t_{k_0,n}^*+\eps_n)^p + \sum_{i=k_0+1}^n |t_{i,n}^*|^p\,.
\]
A similar estimate holds for the integral over the negative half-axis, where we obtain
\[
n\int_{-\infty}^0 |t|^pf_n(t)\,\dint t \leq |t_{k_0-1,n}^*-\varepsilon_n|^p + \sum_{i=1}^{k_0-2}|t_{i,n}^*|^p.
\]
Taking both estimates together, we arrive at the upper bound
\begin{align*}
n\int_{\R} |t|^p f_n(t) \dint t &\leq  \sum_{i=1}^n |t_{i,n}^*|^p + (t_{k_0,n}^*+\eps_n)^p + |t_{k_0-1,n}^*-\eps_n|^p\\
&=
1 + (t_{k_0,n}^*+\eps_n)^p + |t_{k_0-1,n}^*-\eps_n|^p.
\end{align*}
It remains to show that the second and the third summand tend to $0$, as $n\to\infty$. In fact, we shall even prove that $|t_{1,n}^*|=o(1)$  and $t_{n,n}^* = o(1)$, from which the claim follows, since $\eps_n\to 0$. Assume, by contraposition, that there are infinitely many $n$'s for which
\[
\max\{|t_{1,n}^*|, t_{n,n}^*\} > 2\delta
\]
for some $\delta\in(0,1)$. In the following, we restrict $n$ to the subsequence for which the above always holds.
Without loss of generality let $t_{n,n}^* \geq  |t_{1,n}^*|$. By definition of $\log \Delta_n(p)$ given in~\eqref{eq:log_Delta_n}, we have
\[
\frac{2}{(n-1)(n-2)} \sum_{1\leq i < j \leq n-1} \log |t_{i,n}^* -t_{j,n}^*| - \frac 1p \log \left( \frac 1 {n-1}\sum_{i=1}^{n-1} |t_{i,n}^*|^p \right) \leq \log \Delta_{n-1}(p).
\]
On the other hand,
\[
\frac {1}{n-1}\sum_{i=1}^{n-1} |t_{i,n}^*|^p = \frac{1}{n-1} \left(\sum_{i=1}^n |t_{i,n}^*|^p - |t_{n,n}^*|^p\right)
\leq
\frac{1-2\delta}{n-1}
\leq
\frac {1-\delta}n
\]
if $n$ is sufficiently large. It follows that
\[
\sum_{1\leq i < j \leq n-1} \log |t_{i,n}^* -t_{j,n}^*| \leq \frac{(n-1)(n-2)}{2} \left(\log \Delta_{n-1}(p) + \frac 1p \log \frac{1-\delta}n\right).
\]
Further, the trivial bound $|t_{i,n}^*|\leq 1$ for all $1\leq i\leq n$ implies that
\[
\sum_{1\leq i \leq n-1} \log|t_{n,n}^* - t_{i,n}^*| \leq n \log 2.
\]
Taking the sum of the last two inequalities, we arrive at
\[
\sum_{1\leq i < j \leq n} \log |t_{i,n}^* -t_{j,n}^*| \leq \frac{(n-1)(n-2)}{2} \left(\log \Delta_{n-1}(p) + \frac 1p \log \frac{1-\delta}{n}\right) + n\log 2.
\]
Recalling that $(t_{1,n}^*,\ldots, t_{n,n}^*)$ is a maximizer of $\log\Delta_n(p)$ given in~\eqref{eq:log_Delta_n}, we obtain from the previous estimate that
\begin{align*}
\log \Delta_n(p)
&=
\frac 2 {n(n-1)} \sum_{1\leq i < j \leq n} \log |t_{i,n}^* -t_{j,n}^*|  - \frac 1p \log \frac 1n \\
&\leq
\frac{n-2}{n} \left(\log \Delta_{n-1}(p) + \frac 1p \log \frac{1-\delta}{n}\right) + \frac {2\log 2}{n-1}- \frac 1p \log \frac 1n.
\end{align*}
As $n\to\infty$, the left-hand side tends to $\log \Delta(p)$, whereas the expression on the right-hand side tends to $\log \Delta(p) + \frac{1}{p}\log (1-\delta) < \log \Delta(p)$. This contradiction completes the proof of~\eqref{eq:limsup_p_moment_f_n}.

\vspace*{2mm}
\noindent
\textit{Proof of~\eqref{eq:log_energy_f_n}.}
We split the double integral on the left-hand side of~\eqref{eq:log_energy_f_n} into a double sum as follows:
\[
\int_\R\int_\R f_n(x) f_n(y) \log |x-y|\, \dint x\, \dint y =  \frac {1}{n^2\eps_n^2}\sum_{i=1}^n\sum_{j=1}^n \int_{B_{i,n}}\int_{B_{j,n}} \log |x-y| \,\dint x\, \dint y.
\]
Observe that each summand on the right-hand side represents the interaction between the one-sided neighborhoods of $t_{i,n}^*$ and $t_{j,n}^*$.

\vspace*{2mm}
\noindent
\textit{Case 1: Self-interactions.} Let us take some $i>k_0$ and consider  the interaction of the neighborhood $B_{i,n} = [t_{i,n}^* -\eps_n, t_{i,n}^*]$ with itself. If we denote by $X$ and $Y$ two independent random variables with uniform distribution on the interval $[0,1]$, then $t_{i,n}^* -\eps_n X$ and $t_{i,n}^* -\eps_n Y$ are uniformly distributed on the corresponding neighborhood $B_{i,n}$ and we can write
\begin{align*}
\frac 1 {n^2 \eps_n^2} \int_{B_{i,n}} \int_{B_{i,n}} \log |x-y| \,\dint x\, \dint y
&=
\frac {\E \log |(t_{i,n}^* -\eps_n X) - (t_{i,n}^* -\eps_n Y)|}{n^2}\\
&=
\frac {\log \eps_n + \E \log |X-Y|}{n^2}
=
O\left(\frac{\log n} {n^2}\right)
\end{align*}
because $\E \log |X-Y|$ is finite and $\eps_n = n^{-2 c_p}$.
Similar estimates hold in the case $i<k_0-1$ and in the exceptional cases $i=k_0, k_0-1$. For the sum of self-interactions, we obtain the bound
\[
\sum_{i=1}^n \frac 1 {n^2 \eps_n^2} \int_{B_{i,n}} \int_{B_{i,n}} \log |x-y|\, \dint x\, \dint y = O\left(\frac{\log n} {n}\right)=o(1).
\]

\vspace*{2mm}
\noindent
\textit{Case 2: Interactions between different intervals.} Let $i\neq j$. Then the corresponding interval $B_{i,n}$ (and, similarly,  $B_{j,n}$) has either the form $[t_{i,n}^* -\eps_n, t_{i,n}^*]$ or the form $[t_{i,n}^*, t_{i,n}^*+\eps_n]$. If $X$ and $Y$ are again independent random variables each with uniform distribution on the interval $[0,1]$, then the uniformly distributed random variables on $B_{i,n}$ and $B_{j,n}$ have the form $t^*_{i,n} \pm \eps_n X$ and $t^*_{j,n} \pm \eps_n Y$ for an appropriate choice of signs. Thus,
\begin{align*}
\frac 1 {n^2 \eps_n^2} \int_{B_{i,n}} \int_{B_{j,n}} \log |x-y| \,\dint x\, \dint y
&=
\frac {\E \log |(t_{i,n}^* \pm \eps_n X) - (t_{j,n}^* \pm \eps_n Y)|}{n^2} \\
&=
\frac 1 {n^2} \log |t_{i,n}^*  - t_{j,n}^*| +  \frac 1{n^2}\, \E \log \left|1 +  \eps_n  \frac{ \pm X \mp Y }{t_{i,n}^*  - t_{j,n}^*} \right|.
\end{align*}
To estimate the second term on the right-hand side, we recall that in all cases except when $(i,j)=(k_0-1,k_0)$ or $(i,j)=(k_0,k_0-1)$, we have $|t_{i,n}^*-t_{j,n}^*| > n^{-c_p}$, whereas $\eps_n = n^{-2c_p}$, and therefore
\[
\frac 1{n^2}\, \E \log \left|1 +  \eps_n  \frac{ \pm X \mp Y }{t_{i,n}^*  - t_{j,n}^*} \right|
=
\frac 1{n^2}  O\left(  \frac{\eps_n}{|t_{i,n}^*  - t_{j,n}^*|} \right)
=
O\left(\frac1{n^{2+c_p}}\right).
\]
This estimate is uniform in $i,j$ and the sum of at most $n^2$ error terms of the above form is $o(1)$.
In the exceptional case when $(i,j)=(k_0-1,k_0)$, we observe that the intervals were chosen so that the corresponding term has the form
\[
\frac 1 {n^2 \eps_n^2} \int_{B_{k_0-1,n}} \int_{B_{k_0,n}} \log |x-y|\, \dint x\, \dint y
\geq
\frac 1 {n^2}\, \log |t_{k_0-1,n}^*  - t_{k_0,n}^*|,
\]
because $|x-y|\geq |t_{k_0-1,n}^*  - t_{k_0,n}^*|$ for $x\in B_{k_0-1, n}$, $y\in B_{k_0, n}$. The same estimate holds if $(i,j)=(k_0,k_0-1)$.

\vspace*{2mm}
Taking all the estimates of Case 1 and Case 2 together, we arrive at
$$
\int_\R\int_\R f_n(x) f_n(y) \log |x-y|\, \dint x\, \dint y  \geq
\frac{2}{n^2} \sum_{1\leq i < j \leq n} \log |t_{i,n}^* - t_{j,n}^*| -o(1),
$$
which completes the proof of~\eqref{eq:log_energy_f_n}.
\end{proof}

\subsection{Asymptotic behaviour of $I_{n,\beta,p}$}

We continue with the asymptotic behavior of the quantity $I_{n,\beta,p}$ for $0 < p < \infty$, which was defined in~\eqref{eq:I_n_beta_p}. Again we shall work with a general parameter $\beta\in(0,\infty)$.

\begin{lemma}\label{lem:asymptoticsIp<infinity}
Let $0 < p < \infty$ and $\beta\in(0,\infty)$. Then, as $n\to\infty$,
\[
I_{n,\beta,p}^{2/n^2} \asymptequiv n^{-\frac{\beta}{p}}\Delta^\beta(p).
\]
\end{lemma}
\begin{proof}
We start with the upper bound. Using the definition of $\Delta_n(p)$ given in~\eqref{eq:def delta_n(p)} and the fact that $\B_p^n\subset \B_\infty^n$, we obtain
\begin{align*}
I_{n,\beta,p} & = \int_{\B_p^n} \prod_{1\leq i < j \leq n} |t_i-t_j|^\beta\, \dint t_1\dots\dint t_n \cr
& \leq \vol_n(\B_\infty^n)\, \Big(\Delta_n(p)\Big)^{\beta\frac{n(n-1)}{2}}n^{-\frac{\beta}{p}\frac{n(n-1)}{2}} \cr
& =2^n\Big(\Delta_n(p)\Big)^{\beta\frac{n(n-1)}{2}}n^{-\frac{\beta}{p}\frac{n(n-1)}{2}}.
\end{align*}
Therefore,
\[
\limsup_{n\to\infty} n^{\beta/p}I_{n,\beta,p}^{2/n^2}  \leq \Delta^\beta(p).
\]
Let us continue with the lower bound. To this end, we consider the maximizers $t_{1,n}^*,\dots,t_{n,n}^*$ from Lemma \ref{lem: delta_n(p)}. Similar to the proof of Lemma \ref{lem:asymptotic I infinity}, we consider small neighbourhoods around these points. More precisely, let us define $m:=\min\{m_1,m_2\}$ (see Lemma \ref{lem: delta_n(p)} for the definition of $m_1$ and $m_2$) and consider for some small $\varepsilon \in (0,1/2)$ the one-sided neighborhoods
\[
[t_{k_0-1,n}^*-\varepsilon m, t_{k_0-1,n}^*] \qquad\text{and}\qquad [t_{k_0,n}^*, t_{k_0,n}^*+\varepsilon m],
\]
around $t_{k_0-1,n}^*$ and $t_{k_0,n}^*$ as well as the two-sided neighbourhoods
\[
[t_{i,n}^* - \varepsilon m , t_{i,n}^*+\varepsilon m ],
\]
around $t_{i,n}^*$ for $i\neq k_{0}-1,k_0$. Consider the $n$-dimensional box
\[
Q:= \prod_{i=1}^{k_0-2}[t_{i,n}^*-\varepsilon m,t_{i,n}^*+\varepsilon m ] \times [t_{k_0-1,n}^*-\varepsilon m, t_{k_0-1,n}^*] \times [t_{k_0,n}^*, t_{k_0,n}^*+\varepsilon m] \times \prod_{i=k_0+1}^n [t_{i,n}^*-\varepsilon m,t_{i,n}^*+\varepsilon m ],
\]
which satisfies $\vol_n(Q) = (\varepsilon m)^2(2\varepsilon m)^{n-2}$. Then, for all $(t_1,\dots,t_n)\in Q$,
\begin{align}\label{ineq: lower bound product on Q in p case}
\prod_{1\leq i < j \leq n} |t_i-t_j| \geq (1-2\varepsilon)^{\frac{n(n-1)}{2}}\bigg( \frac{\Delta_n(p)}{n^{\frac{1}{p}}}\bigg)^{\frac{n(n-1)}{2}},
\end{align}
which follows along the same lines as the corresponding part in the proof of Lemma \ref{lem:asymptotic I infinity} except for the following observation:  even though no estimate on $t^*_{k_0,n} - t^*_{k_0-1,n}$ is available, the way the one-sided neighborhoods were chosen allows us to write 
\[|t_{k_0} - t_{k_0-1}| \geq  |t^*_{k_0,n} - t^*_{k_0-1,n}|\geq (1-2\eps) |t^*_{k_0,n} - t^*_{k_0-1,n}|.
\] 
Putting things together, we obtain
\begin{align*}
I_{n,\beta,p} & = \int_{\B_p^n} \prod_{1\leq i < j \leq n} |t_i-t_j|^\beta \dint t_1\dots\dint t_n \cr
& \geq \vol_n(Q) (1-2\varepsilon)^{\beta\frac{n(n-1)}{2}}\bigg( \frac{\Delta_n(p)}{n^{\frac{1}{p}}}\bigg)^{\beta\frac{n(n-1)}{2}}.
\end{align*}
Choosing $\varepsilon:=\frac{1}{n}$ and recalling that $m\geq n^{-2-2/p}$ from Lemma \ref{lem: delta_n(p)}, we conclude that
\[
\liminf_{n\to\infty} n^{\beta/p}I_{n,\beta,p}^{2/n^2} \geq \Delta^\beta(p).
\]
This completes the proof.
\end{proof}

\subsection{Proof of Theorem \ref{thm:VolumeAsymptotics}}

Recalling that
\[
\vol_{\beta,n}(\B_{p,\beta}^n) = c_{n,\beta} I_{n,\beta,p}
\]
for $0<p\leq\infty$ and $\beta\in\{1,2,4\}$, the proof of the asymptotic formula for the volume of unit balls in the classical matrix ensembles is now a simple consequence of the results we obtained in the previous sections. If $p=\infty$, the result follows by combining Lemma \ref{lem:asymptotic I infinity} with Lemma \ref{lem:AsymptoticConstant}. If otherwise $0 < p<\infty$, the result is a consequence of Lemma \ref{lem:asymptoticsIp<infinity} and again Lemma \ref{lem:AsymptoticConstant}. Finally, the explicit value of $\Delta(p)$ follows from Theorem~\ref{lem:limit delta_n(p)} and the remark thereafter. This completes the proof of the theorem. \hfill $\Box$

\section{Random sampling in matrix balls \& a weak law of large numbers} \label{sec:sampling and weak law}

We shall present in this section the probabilistic ingredients that we need to study the asymptotic volume of intersections of unit balls in the matrix ensembles $\mathscr{H}_n(\mathbb{F}_\beta)$ . We start with a probabilistic representation of the volume measure on $\B_{p,\beta}^n$ and then present a limit theorem for the empirical eigenvalue distribution of matrices in our ensembles. The latter two are then used to prove a weak law of large numbers for the eigenvalues of a matrix chosen uniformly at random from $\B_{p,\beta}^n$.

\subsection{Random sampling in $\B_{p,\beta}^n$}

We start by recalling the joint law of the $n$ real eigenvalues $\lambda_1(Z) \leq \ldots\leq \lambda_n(Z)$ of an $n\times n$ matrix $Z$ uniformly distributed in $\B_{p,\beta}^n$ with $\beta\in\{1,2,4\}$. The following result follows easily from the Weyl integration formula; see Lemma~\ref{lem:IntegrationFormulaForEnsembles}. 

\begin{lemma}\label{lem:distribution in random order}
Let $0 < p \leq \infty$, $\beta\in\{1,2,4\}$ and $Z$ be a matrix chosen uniformly at random in $\B^n_{p,\beta}$. Then, for any $B\in\mathscr B(\R^n)$,
\begin{align*}
& \Pro\Big((\lambda_{1}(Z),\dots,\lambda_{n}(Z))\in B\Big) = C_{p,\beta,n}\,\int_{B\cap\B_p^n}   \mathbbm 1_{\{y\in\R^n:y_1<\ldots<y_n\}}(x)\,\prod_{1\leq i < j \leq n}|x_i-x_j|^{\beta} \,\dint x,
\end{align*}
where $C_{p,\beta,n}\in(0,\infty)$ is a suitable normalization constant. Moreover, if $\pi$ is a uniform random permutation in $\mathfrak S(n)$, which is independent from $Z$, then, for any $B\in\mathscr B(\R^n)$,
\begin{align*}
& \Pro\Big((\lambda_{\pi(1)}(Z),\dots,\lambda_{\pi(n)}(Z))\in B\Big) = \frac {C_{p,\beta,n}} {n!} \,\int_{B\cap\B_p^n}   \prod_{1\leq i < j \leq n}|x_i-x_j|^{\beta} \,\dint x.
\end{align*}
\end{lemma}

For $m\geq 0$ we let $f:\R^n\to [0,\infty)$ be a function satisfying $f(tx)=t^mf(x)$ for all $t\geq 0$ and say that $f$ is homogeneous of degree $m$. We also assume that $f$ is integrable with respect to the cone probability measure $\mu_{\B_p^n}$ on $\SSS_p^{n-1}$. In what follows, we shall write $\mathscr F_m^+(\R^n)$ for the class of such $m$-homogeneous, non-negative and integrable functions.

We shall now prove a Schechtman-Zinn type probabilistic representation, where we follow a different route compared to \cite{SchechtmanZinn}. In fact, our argument will be based on the polar integration formula \eqref{eq:polar integration} and not on a limiting argument. The next lemma also shows that if we multiply a random vector $X/\|X\|_p$ on the boundary of an $\B_p^n$ which has a density proportional to $f$ (with respect to the cone probability measure) with a (properly normalized) uniformly distributed random variable, then the resulting random vector still has density proportional to $f$, but now with respect to the $n$-dimensional Lebesgue measure on $\B_p^n$ instead of the cone probability measure.

\begin{lemma}\label{lem:schechtman-zinn cone and volume measure}
Let $0 <  p <\infty$ and $f\in\mathscr F_m^+(\R^n)$ for some $m \geq 0$. Let the random variables $X_1,\dots,X_n$ have joint density on $\R^n$ given by
\[
x\mapsto C_{p,f,n} e^{-\|x\|_p^p} f(x)
\]
with respect to the Lebesgue measure, where $C_{p,f,n}\in(0,\infty)$ is a suitable normalization constant. Define $X:=(X_1,\dots,X_n)$. Then the random vector
$$
X\over \|X\|_p
$$
has density $x\mapsto c_{p,f,n}f(x)$ with respect to the cone probability measure on $\SSS_p^{n-1}$, where $c_{p,f,n}$ is a normalization constant. Moreover, $X/\|X\|_p$ is independent from $\|X\|_p$.

In addition, if $U$ is a random variable uniformly distributed on $[0,1]$ and independent from $X$, then
\[
U^{\frac{1}{n+m}}\frac{X}{\|X\|_p}
\]
has density $x\mapsto{c}_{p,f,n}^{(2)}f(x)$ with respect to the Lebesgue measure on $\B_p^n$, where ${c}_{p,f,n}^{(2)}\in(0,\infty)$ is another normalization constant.
\end{lemma}
\begin{proof}
We shall use the polar integration formula \eqref{eq:polar integration}. Let $X:=(X_1,\dots,X_n)$ be a random vector with joint density
\[
\R^n\to\R,\quad x\mapsto  C_{p,f,n} f(x) e^{-\|x\|_p^p}\,,
\]
where $f\in \mathscr F_m^+(\R^n)$ and $C_{p,f,n}\in(0,\infty)$ is a suitable normalization constant. Consider two arbitrary non-negative measurable functions
\[
h: \SSS_p^{n-1} \to [0,\infty) \qquad\text{and}\qquad g:[0,\infty)\to[0,\infty).
\]
Then, using the polar integration formula \eqref{eq:polar integration} and the homogeneity of $f$, we obtain
\begin{align*}
&\E \bigg[h\bigg(\frac{X}{\|X\|_p}\bigg) g(\|X\|_p)\bigg] \cr
& = \int_{\R^n} C_{p,f,n} f(x) e^{-\|x\|_p^p} h\bigg(\frac{x}{\|x\|_p}\bigg) g(\|x\|_p) \,\dint x \cr
& =n\vol_n(\B_p^n)
\int_0^\infty r^{n-1} \bigg( \int_{\SSS_p^{n-1}} C_{p,f,n} f(ry)e^{-r^p\|y\|_p^p} h\bigg(\frac{y}{\|y\|_p}\bigg) g(\|ry\|_p) \, \mu_{\B_p^n}(\dint y)\bigg) \,\dint r \cr
& = \vol_n(\B_p^n) \bigg(\int_{\SSS_p^{n-1}}f(y)h(y)\, \mu_{\B_p^n}(\dint y)\bigg)\bigg(\int_0^\infty C_{p,f,n} n r^{m+n-1} e^{-r^p}g(r)\,\dint r\bigg).
\end{align*}
This immediately implies two things. First, we read off from the product structure that $X/\|X\|_p$ and $\|X\|_p$ are independent. Second, choosing $g\equiv 1$, we conclude that $X/\|X\|_p$ has a density of the form $x\mapsto c_{p,f,n}f(x)$ with respect to the cone measure $\mu_{\B_p^n}$, where $c_{p,f,n}\in(0,\infty)$ is a suitable normalization constant. Moreover, if $U$ is a random variable uniformly distributed on $[0,1]$ and independent from $X$, then for every measurable function $\widetilde h: \B_p^n \to [0,\infty)$, we have
\begin{align*}
\E \bigg[\widetilde h\Big(U^{1\over n+m}{X\over \|X\|_p}\Big)\bigg]
&=\int_0^1\int_{\SSS_p^{n-1}} \widetilde h(ry) c_{p,f,n}f(y)(n+m)r^{m+n-1}\,\mu_{\B_p^n}(\dint y)\,\dint r\\
&={c}_{p,f,n}^{(1)}\int_0^1\int_{\SSS_p^{n-1}} \widetilde h(ry)f(ry)r^{n-1}\,\mu_{\B_p^n}(\dint y)\,\dint r\\
&=
{c}_{p,f,n}^{(2)}\int_{\B_p^n} \widetilde h(x)f(x)\,\dint x,
\end{align*}
with suitable constants ${c}_{p,f,n}^{(1)}$ and ${c}_{p,f,n}^{(2)}$, where we used once more the polar integration formula \eqref{eq:polar integration} and the homogeneity of $f$. Therefore, the random vector
\[
U^{\frac{1}{n+m}}\frac{X}{\|X\|_p}
\]
has density $x\mapsto {c}_{p,f,n}^{(2)} f(x)$ with respect to the Lebesgue measure on $\B_p^n$.
\end{proof}

As a corollary to the previous lemmas, we obtain the following Schechtman-Zinn type probabilistic representation for the eigenvalues of a matrix sampled uniformly from the unit ball $\B_{p,\beta}^n$. It follows directly by combining Lemma \ref{lem:distribution in random order} with Lemma \ref{lem:schechtman-zinn cone and volume measure} and by taking $f(x)=\prod_{1\leq i<j\leq n}|x_i-x_j|^\beta$, which belongs to the class $\mathscr{F}^+_m(\R^n)$ for $m=\beta n(n-1)/2$.

\begin{cor}\label{cor:EigenvaluesProbabRep}
Let $0 <  p < \infty$ and $Z$ be uniformly distributed in $\B_{p,\beta}^n$, $\beta\in\{1,2,4\}$. Consider a permutation $\pi$ uniformly distributed on $\mathfrak{S}(n)$, which is independent from $Z$. Then
\[
\big(\lambda_{\pi(1)}(Z),\ldots,\lambda_{\pi(n)}(Z)\big) \stackrel{d}{=}U^{1\over n+m}{X\over\|X\|_p}\qquad\text{with}\qquad m={\beta n(n-1)\over 2},
\]
where $U$ is uniformly distributed on $[0,1]$ and, independently of $U$, the vector $X=(X_1,\ldots,X_n)$ has joint density with respect to Lebesgue measure on $\R^n$ which is proportional to
\[
e^{-\sum\limits_{i=1}^n|x_i|^p}\prod_{1\leq i<j\leq n}|x_i-x_j|^\beta,\qquad x\in\R^n,
\]
where the proportionality constant only depends on $\beta$, $p$ and $n$.
\end{cor}

\subsection{A convergence result from random matrix theory}\label{sec:RMT}

It is well known (see~\cite[Theorem~2.5.2]{AGZ2010} and~\cite[Chapter 2]{PS2011}) that the joint distribution of the eigenvalues $\lambda_1^{(n)}\leq \ldots\leq \lambda_n^{(n)}$ of a standard Gaussian random matrix $Z_n$ from $\mathscr H_n(\mathbb F_\beta)$ (with $\beta=1$ corresponding to GOE, $\beta=2$  corresponding to GUE and $\beta=4$ corresponding to GSE, see~\cite[pp.~188--189 and p.~51]{AGZ2010} for the definition of the standard Gaussian distribution on $\mathscr H_n(\mathbb F_\beta)$) has density proportional to
$$
e^{-{\beta\over 4}\sum_{i=1}^n \lambda_i^2} \left(\prod_{1\leq i<j\leq n}|\lambda_i-\lambda_j|^\beta \right)\, \mathbbm 1_{\lambda_1 < \ldots < \lambda_n},
$$
(the proportionality constant can explicitly be computed using Selberg's integral formula, see, e.g.,\ \cite[Theorem~2.5.8]{AGZ2010}). We also define the empirical distribution of scaled eigenvalues of $Z_n$ as the random measure $\rho_n$ on $\R$ given by
$$
\rho_n:={1\over n}\sum_{i=1}^n\delta_{\lambda_i^{(n)}/\sqrt n},
$$
where $\delta_x$ stands for the Dirac measure at $x$. It is a well known fact (see \cite[Chapter 2]{AGZ2010} and~\cite[Chapter 2]{PS2011}) that, with probability one,
$$
\rho_n\stackrel{\text{w}}{\longrightarrow}\rho,
$$
where $\rho$ is the Wigner semicircular distribution on $\R$ with density given by $x\mapsto{1\over 2\pi}\sqrt{4-x^2}$, $|x|\leq 2$.

More generally, let us consider a model for a random $n\times n$ matrix whose joint eigenvalue distribution on $\R^n$ (where, as above, the eigenvalues are ordered increasingly) is absolutely continuous with respect to the Lebesgue measure and has density proportional to
\begin{equation}\label{eq:matrix_model}
e^{-{n\beta\over 2}\sum_{i=1}^nV(x_i)} \left(\prod_{1\leq i<j\leq n}|x_i-x_j|^\beta \right) \mathbbm 1_{x_1 <\ldots < x_n}\qquad\text{with}\qquad V(x)={|x|^p\over p}
\end{equation}
for some $p >0$. The associated empirical eigenvalue distribution will be denoted by $\rho_n^{(p)}$, that is,
$$
\rho_n^{(p)} := {1\over n}\sum_{i=1}^n\delta_{\lambda_i^{(n)}}.
$$
%
The following fact is a special case of \cite[Corollary 1.2]{HardyLDP} combined with the explicit formulas taken from \cite[page 364]{PS2011}. Alternatively, it can be deduced from the large deviation principle for $\rho_n^{(p)}$ stated in~\cite[Theorem 5.4.3]{hiai_petz} in conjunction with the characterization of the scaled Ullman distribution as the unique minimizer of the information function~\cite[Proposition 5.3.4]{hiai_petz}.

\begin{lemma}\label{lem:EmpiricalEVRMT}
Let $0<p<\infty$. Then, as $n\to\infty$, one has that, almost surely,
$$
\rho_n^{(p)} \stackrel{\text{w}}{\longrightarrow}\rho^{(p)},
$$
where the limit measure $\rho^{(p)}$ is deterministic and has the following rescaled Ullman density:
$$
g^{(p)}(x):={1\over b_p}h_p\Big({x\over b_p}\Big),\qquad |x|\leq b_p,
$$
with $h_p(x)$, $|x|\leq 1$, and the constant $b_p$ given by
$$
h_p(x):={p\over\pi}\int_{|x|}^1{t^{p-1}\over\sqrt{t^2-x^2}}\,\dint t\qquad\text{and}\qquad b_p:=\bigg(\frac{p \sqrt{\pi} \Gamma(\frac p2)}{\Gamma(\frac {p+1}2)}\bigg)^{1/p}.
$$
\end{lemma}

\begin{rmk}
We notice that in the particular case $p=2$ we get back the Wigner semicircle distribution with density $g^{(2)}(x)={1\over2\pi}\sqrt{4-x^2}$, $|x|\leq 2$, mentioned at the beginning of this paragraph in connection with the Gaussian ensembles GOE, GUE and GSE.
\end{rmk}

\begin{cor}\label{cor:EmpiricalDistributionX}
Let $0<p<\infty$ and let $X=(X_1,\ldots,X_n)$ be a random vector with distribution as described in Corollary~\ref{cor:EigenvaluesProbabRep}. Then, as $n\to\infty$, one has that, almost surely,
\begin{equation}\label{eq:conv_empirical_to_zeta}
\frac{1}{n}
\sum_{i=1}^n \delta_{n^{-1/p} X_i} \overset{\text{w}}{\longrightarrow}\zeta^{(p)}
\end{equation}
with $\zeta^{(p)}$ being a measure with rescaled Ullman density of the form ${1\over c_{\beta,p}}h_p ({x\over c_{\beta,p}})$, $|x|<c_{\beta,p}$, where the exact value of the constant $c_{\beta,p}\in(0,\infty)$ is not important in what follows.
\end{cor}
\begin{proof}
If $(X_1,\ldots,X_n)$ is as in Corollary~\ref{cor:EigenvaluesProbabRep}, then the joint density of $\big({n\beta\over 2p}\big)^{-1/p}(X_1,\ldots,X_n)$ is proportional to the expression given in~\eqref{eq:matrix_model} except that no indicator function is needed.
By Lemma \ref{lem:EmpiricalEVRMT}, with probability $1$ we have
$$
\frac{1}{n}
\sum_{i=1}^n \delta_{\big({n\beta\over 2p}\big)^{-1/p} X_i} \overset{\text{w}}{\longrightarrow}\rho^{(p)}.
$$
This differs from the claimed convergence~\eqref{eq:conv_empirical_to_zeta} just by a rescaling determined by $\beta$ and $p$.
\end{proof}

\subsection{A weak law of large numbers}\label{subsec:weak LLN}

We can now prove the following weak law of large numbers.

\begin{thm}\label{thm:wlln}
Let $0 < p ,q <\infty$ and $Z_n$ be uniformly distributed on $\B^n_{p,\beta}$, $\beta\in\{1,2,4\}$. Then, as $n\to\infty$,
\[
n^{1/p-1/q}\,\Big(\sum_{i=1}^n |\lambda_i(Z_n)|^q\Big)^{1/q}\stackrel{\Pro}{\longrightarrow}
C_{p,q}
=
\frac{\left(\frac{\Gamma(\frac{q+1}{2})}{2\sqrt{\pi}\Gamma(\frac{q+2}{2})}\right)^{1/q}}{\left(\frac{\Gamma(\frac{p+1}{2})}{2\sqrt{\pi}\Gamma(\frac{p+2}{2})}\right)^{1/p}}\,.
\]
\end{thm}
\begin{proof}
Let $U$ be uniformly distributed on $[0,1]$ and $X$ be a random vector with density as described in Corollary~\ref{cor:EigenvaluesProbabRep}. We assume that $U$ and $X$ are independent. By Corollary~\ref{cor:EigenvaluesProbabRep} we have that
$$
\sum_{i=1}^n |\lambda_i(Z_n)|^q \overset{\text{d}}{=}U^{q\over n+m}{\sum_{i=1}^n|X_i|^q\over\big(\sum_{i=1}^n|X_i|^p\big)^{q/p}}.
$$
Thus,
\begin{align*}
\Big(\sum_{i=1}^n |\lambda_i(Z_n)|^q\Big)^{1/q} &\overset{\text{d}}{=}U^{1\over n+m}{\big(\sum_{i=1}^n|X_i|^q\big)^{1/q}\over\big(\sum_{i=1}^n|X_i|^p\big)^{1/p}} \\
&= U^{1\over n+m}{\big({1\over n}\sum_{i=1}^n(n^{-1/p}|X_i|)^q\big)^{1/q}\over\big({1\over n}\sum_{i=1}^n(n^{-1/p}|X_i|)^p\big)^{1/p}}\,n^{{1\over q}-{1\over p}}.
\end{align*}
Defining the random probability measure
$$
\xi_n:={1\over n}\sum_{i=1}^n\delta_{n^{-1/p}X_i},
$$
on the Borel sets of $\R$, we obtain
\[
\Big(\sum_{i=1}^n |\lambda_i(Z_n)|^q\Big)^{1/q} \overset{\text{d}}{=} U^{1\over n+m}{\big(\int_\R |x|^q\,\xi_n(\dint x)\big)^{1/q}\over\big(\int_\R |x|^p\,\xi_n(\dint x)\big)^{1/p}}\,n^{{1\over q}-{1\over p}}.
\]
With the help of  Corollary \ref{cor:EmpiricalDistributionX} we are going to prove that, as $n\to\infty$,
\begin{equation}\label{eq:WeakConvergencePP}
Y_n:=\int_\R |x|^r\,\xi_n(\dint x)\overset{\text{d}}{\longrightarrow}\int_\R |x|^r\,\zeta^{(p)}(\dint x)=:Y
\end{equation}
for any $r>0$. Indeed, from \cite[Theorem 11.1.2 (i)]{PS2011} we know that there exists a constant $K_{\beta,p}\in(0,\infty)$ such that the intensity measure $\E\xi_n$ of $\xi_n$ has a Lebesgue density on $\R$, which is bounded by $e^{-C_{\beta,p}n|x|^p}$ whenever $|x|>K_{\beta,p}$, where $C_{\beta,p}\in(0,\infty)$ is another constant. Let us define $L=L_{\beta,p}:=\max\{K_{\beta,p},c_{\beta,p}\}$ with the constant $c_{\beta,p}$ as in Corollary \ref{cor:EmpiricalDistributionX} (recall that the interval $(-c_{\beta,p},c_{\beta,p})$ is the support of $\zeta^{(p)}$) as well as the random variables
\[
Y_n^{(1)} := \int_{-L}^L |x|^r\,\xi_n(\dint x)\qquad\text{and}\qquad Y_n^{(2)}:=\int_{\R\setminus(-L,L)} |x|^r\,\xi_n(\dint x).
\]
Then we conclude $Y_n^{(1)}\overset{\text{d}}{\longrightarrow}Y$ from Corollary \ref{cor:EmpiricalDistributionX}, since the function $x\mapsto|x|^r$ is bounded on $(-L,L)$. On the other hand, $x\mapsto |x|^r$ is non-negative and measurable, and therefore it follows from the previously mentioned exponential upper bound and the dominated convergence theorem that $Y_n^{(2)}\overset{\text{d}}{\longrightarrow}0$. In fact,
\[
\E Y_n^{(2)} = \int_{\R\setminus(-L,L)} |x|^r\,\E\xi_n(\dint x)  \leq \int_{\R\setminus(-L,L)}|x|^re^{-C_{\beta,p}n|x|^p}\,\dint x \leq 2\int_0^\infty x^re^{-C_{\beta,p}nx^p}\,\dint x\longrightarrow 0,
\]
as $n\to\infty$, by the dominated convergence theorem (take $n=1$ to get an integrable majorant). As a consequence, the pair $(Y_n^{(1)},Y_n^{(2)})$ converges in distribution, as $n\to\infty$, to the pair $(Y,0)$ and the continuous mapping theorem \cite[Lemma 3.3]{Kallenberg} then yields that $Y_n=Y_n^{(1)}+Y_n^{(2)}\overset{\text{d}}{\longrightarrow}Y$, as $n\to\infty$. This proves \eqref{eq:WeakConvergencePP}.

In combination with the continuous mapping theorem \cite[Lemma 3.3]{Kallenberg}, we get from \eqref{eq:WeakConvergencePP} that
$$
\Big(\int_\R |x|^q\,\xi_n(\dint x)\Big)^{1/q}\overset{\text{d}}{\longrightarrow}\Big(\int_\R |x|^q\,\zeta^{(p)}(\dint x)\Big)^{1/q}
$$
and
$$
\Big(\int_\R |x|^p\,\xi_n(\dint x)\Big)^{1/p}\overset{\text{d}}{\longrightarrow}\Big(\int_\R |x|^p\,\zeta^{(p)}(\dint x)\Big)^{1/p},
$$
as $n\to\infty$. Since both limiting random variables are constant, this convergence also holds in probability according to Lemma \ref{lem:ConvergenceDistributionProbability}, that is,
$$
\Big(\int_\R |x|^q\,\xi_n(\dint x)\Big)^{1/q}\overset{\Pro}{\longrightarrow}\Big(\int_\R |x|^q\,\zeta^{(p)}(\dint x)\Big)^{1/q}
$$
and
$$
\Big(\int_\R |x|^p\,\xi_n(\dint x)\Big)^{1/p}\overset{\Pro}{\longrightarrow}\Big(\int_\R |x|^p\,\zeta^{(p)}(\dint x)\Big)^{1/p},
$$
as $n\to\infty$. Moreover, $\Big(\int_\R |x|^p\,\zeta^{(p)}(\dint x)\Big)^{1/p}\neq 0$ and thus Lemma \ref{lem:ConvergenceProductsQuotients} (ii) implies that, as $n\to\infty$,
$$
{\big(\int_\R |x|^q\,\xi_n(\dint x)\big)^{1/q}\over\big(\int_\R |x|^p\,\xi_n(\dint x)\big)^{1/p}} \overset{\Pro}{\longrightarrow}{\big(\int_\R |x|^q\,\zeta^{(p)}(\dint x)\big)^{1/q}\over\big(\int_\R |x|^p\,\zeta^{(p)}(\dint x)\big)^{1/p}}.
$$
Finally, we notice that $U^{1\over n+m}$ converges in probability to $1$, as $n\to\infty$. In view of Lemma \ref{lem:ConvergenceProductsQuotients} (i) this proves that
$$
n^{1/p-1/q}\,\Big(\sum_{i=1}^n |\lambda_i(Z)|^q\Big)^{1/q}\stackrel{\Pro}{\longrightarrow} {\big(\int_\R |x|^q\,\zeta^{(p)}(\dint x)\big)^{1/q}\over\big(\int_\R |x|^p\,\zeta^{(p)}(\dint x)\big)^{1/p}}=:C_{p,q},
$$
as $n\to\infty$. To compute $C_{p,q}$ explicitly, observe that the probability measure $\zeta^{(p)}$ appearing in Corollary~\ref{cor:EmpiricalDistributionX} is the probability distribution of $c_{\beta,p} \mathbb{U}$, where $\mathbb {U}\sim \mathscr U(p)$ is an Ullman random variable. Hence,
\begin{equation}\label{eq:CpqExplicit}
C_{p,q}
=
{\Big(\int_\R |x|^q\,\zeta^{(p)}(\dint x)\Big)^{1/q}\over \Big(\int_\R |x|^p\,\zeta^{(p)}(\dint x)\Big)^{1/p}}
=
\frac{(\E |c_{\beta,p}\mathbb U|^q)^{1/q}}{(\E|c_{\beta,p} \mathbb U|^p)^{1/p}}
=
\frac{(\E |\mathbb U|^q)^{1/q}}{(\E|\mathbb U|^p)^{1/p}}
=
\frac
{\left(\frac{p\Gamma(\frac{q+1}{2})}{(p+q)\sqrt{\pi}\Gamma(\frac{q+2}{2})}\right)^{1/q}}
{\left(\frac{\Gamma(\frac{p+1}{2})}{2\sqrt{\pi}\Gamma(\frac{p+2}{2})}\right)^{1/p}},
\end{equation}
where in the last step we used the formula for the moments of $|\mathbb U|$ stated in~\eqref{eq:expectation ullman power p} and~\eqref{eq:expectation ullman power p_general}.
\end{proof}

\section{Application to high-dimensional intersections}\label{sec:Application}

We prove the result on the intersection of high-dimensional matrix balls in the spirit of Schechtman and Schmuckenschl\"ager that was discussed in the introduction. For $0<p\leq\infty$ and $\beta\in\{1,2,4\}$ we write $\D^n_{p,\beta}$ for the volume normalized matrix balls, that is,
\begin{align*}
\D^n_{p,\beta} &=  \vol_{\beta,n}(\B^n_{p,\beta})^{-2/(n(n-1)\beta + 2n)}\B^n_{p,\beta},
\end{align*}
where we used that the dimension of $\mathscr H_n(\mathbb F_\beta)$ over $\R$ is ${n(n-1)\beta\over 2} + n$.
Having in mind Theorem \ref{thm:VolumeAsymptotics}, we define  $a_{p}(\beta)$ (and similarly $a_{q}(\beta)$) by
$$
a_{p}(\beta) :=
\Delta^\beta(p)\left({4\pi\over \beta}\right)^{\beta/2}e^{3\beta/4}, \quad  0<p\leq\infty.
$$
We shall also need the constants
\begin{equation}\label{eq:Apqbeta}
a_{p,q}:=\bigg(\frac{a_{q}(\beta)}{a_{p}(\beta)}\bigg)^{1/\beta} = \frac{\Delta(q)}{\Delta(p)}
=
\frac{\bigg(\frac{q\sqrt{\pi}\,\Gamma(\frac{q}{2})}{\sqrt{e}\,\Gamma(\frac{q+1}{2})}\bigg)^{1/q}}{\bigg(\frac{p\sqrt{\pi}\,\Gamma(\frac{p}{2})}{\sqrt{e}\,\Gamma(\frac{p+1}{2})}\bigg)^{1/p}}\,.
\end{equation}
We are now able to prove Theorem \ref{thm:ApplInto} discussed in the introduction.
\begin{thm}\label{thm:main_repeat}
Let $0 <  p, q <\infty$ with $p\neq q$ and let $\beta\in\{1,2,4\}$. Then, for $t>0$,
\[
\vol_{\beta, n}(\D^n_{p,\beta}\cap t\, \D^n_{q,\beta}) \stackrel{n\to\infty}{\longrightarrow}
\begin{cases}
0 &: t < e^{\frac{1}{2p} - \frac{1}{2q}} \left(\frac{2p}{p+q}\right)^{1/q} \\
1 &: t > e^{\frac{1}{2p} - \frac{1}{2q}} \left(\frac{2p}{p+q}\right)^{1/q}\,.
\end{cases}
\]
\end{thm}
\begin{proof}
Let $Z_n$ be uniformly distributed on $\B^n_{p,\beta}$ and $X$ be a random vector distributed as in Corollary~\ref{cor:EigenvaluesProbabRep}. From Corollary~\ref{cor:EigenvaluesProbabRep} it follows that
\[
\sum_{i=1}^n|\lambda_i(Z_n)|^q\stackrel{\text{d}}{=} \bigg(U^{\frac{1}{n+m}}\frac{\|X\|_q}{\|X\|_p}\bigg)^{q}.
\]
Consequently, for any sequence $(t_n)_{n\in\N}$ converging to $t>0$, as $n\to\infty$, we have
\begin{align*}
\Pro\Big(Z_n\in t_n a_{p,q}^{-1}n^{1/q-1/p}\, \B^n_{q,\beta}\Big) &= \Pro\Big(n^{1/p-1/q}\,\Big(\sum_{i=1}^n|\lambda_i(Z_n)|^q\Big)^{1/q}  - C_{p,q} \leq a_{p,q}^{-1}\,t_n - C_{p,q}\Big)\\
&  \stackrel{n\to\infty}{\longrightarrow} \begin{cases}
0 &: t < C_{p,q}a_{p,q}\\
1 &: t > C_{p,q}a_{p,q}\,,
\end{cases}
\end{align*}
where we used the weak law of large numbers in Theorem \ref{thm:wlln}. Note that by~\eqref{eq:CpqExplicit} and~\eqref{eq:Apqbeta},
$$
C_{p,q}a_{p,q} = e^{\frac{1}{2p} - \frac{1}{2q}} \left(\frac{2p}{p+q}\right)^{\frac{1}{q}}
$$
is exactly the critical value appearing in the statement of Theorem~\ref{thm:main_repeat}.
On the other hand,
\begin{align*}
\Pro\Big(Z_n\in t_n a_{p,q}^{-1}n^{1/q-1/p}\, \B^n_{q,\beta}\Big)
& \quad = \frac{\vol_{\beta,n}\Big(\{z\in \B^n_{p,\beta}\,:\, \|z\|_q \leq t_n a_{p,q}^{-1}n^{1/q-1/p}  \} \Big)}{\vol_{\beta, n}(\B^n_{p,\beta})} \\
&\quad = \vol_{\beta,n}\bigg(\bigg\{z\in \D^n_{p,\beta}\,:\, \|z\|_q \leq t_n\frac{ a_{p,q}^{-1}n^{1/q-1/p}}{\vol_{\beta, n}(\B^n_{p,\beta})^{2/(n(n-1)\beta + 2n)}}  \bigg\} \bigg)\\
& \quad = \vol_{\beta,n}\bigg(\D^n_{p,\beta} \cap t_n\frac{ a_{p,q}^{-1}n^{1/q-1/p}\vol_{\beta, n}(\B^n_{q,\beta})^{2/(n(n-1)\beta + 2n)}}{\vol_{\beta n^2}(\B^n_{p,\beta})^{2/(n(n-1)\beta + 2n)}}\D^n_{q,\beta} \bigg).
\end{align*}
We now take a sequence $(t_n)_{n\in\N}$ such that
$$
t_n\frac{ a_{p,q}^{-1}n^{1/q-1/p}\vol_{\beta, n}(\B^n_{q,\beta})^{2/(n(n-1)\beta + 2n)}}{\vol_{\beta n^2}(\B^n_{p,\beta})^{2/(n(n-1)\beta + 2n)}} = t.
$$
To complete the proof, we need to show that $\lim_{n\to\infty} t_n =t$. But from Theorem \ref{thm:VolumeAsymptotics} we deduce that
\[
\frac{\vol_{\beta,n}(\B^n_{q,\beta})^{\frac 2 {n(n-1)\beta + 2n}}}{\vol_{\beta,n}(\B^n_{p,\beta})^{\frac 2{n(n-1)\beta + 2n}}}
=
\left(\frac{\vol_{\beta,n}(\B^n_{q,\beta})^{2/\beta n^2}}{\vol_{\beta,n}(\B^n_{p,\beta})^{2/\beta n^2}}\right)^{\frac{\beta n^2}{n(n-1)\beta + 2n}}
 \asymptequiv \frac{a_{q}(\beta)^{1/\beta}}{a_{p}(\beta)^{1/\beta}} n^{1/p-1/q} = a_{p,q} n^{1/p-1/q},
\]
where we used that $n^{\frac{\beta n^2}{n(n-1)\beta + 2n}}\sim n$.  The proof is complete.
\end{proof}

\subsection*{Acknowledgement}

JP has been supported by a \textit{Visiting International Professor (VIP) Fellowship} from the Ruhr University Bochum. ZK and CT were supported by the DFG Scientific Network \textit{Cumulants, Concentration and Superconcentration}.

\bibliographystyle{plain}
\bibliography{ensembles}

\begin{thebibliography}{10}

\bibitem{AGZ2010}
G.~W. Anderson, A.~Guionnet, and O.~Zeitouni.
\newblock {\em An {I}ntroduction to {R}andom {M}atrices}, volume 118 of {\em
  Cambridge Studies in Advanced Mathematics}.
\newblock Cambridge University Press, Cambridge, 2010.

\bibitem{AsymptoticGeometricAnalysisBookPart1}
S.~Artstein-Avidan, A.~Giannopoulos, and V.D. Milman.
\newblock {\em Asymptotic {G}eometric {A}nalysis. {P}art {I}}, volume 202 of
  {\em Mathematical Surveys and Monographs}.
\newblock American Mathematical Society, Providence, RI, 2015.

\bibitem{IsotropicConvexBodies}
S.~Brazitikos, A.~Giannopoulos, P.~Valettas, and B.-H. Vritsiou.
\newblock {\em Geometry of {I}sotropic {C}onvex {B}odies}, volume 196 of {\em
  Mathematical Surveys and Monographs}.
\newblock American Mathematical Society, Providence, RI, 2014.

\bibitem{CK2015}
J.~A. Ch\'avez-Dom\'inguez and D.~Kutzarova.
\newblock Stability of low-rank matrix recovery and its connections to {B}anach
  space geometry.
\newblock {\em J. Math. Anal. Appl.}, 427(1):320 -- 335, 2015.

\bibitem{EF2005}
A.~Eisinberg and G.~Fedele.
\newblock {V}andermonde systems on {G}auss-{L}obatto {C}hebychev nodes.
\newblock {\em Appl. Math. Comput.}, 170:633--647, 2005.

\bibitem{MF1}
M.~Fekete.
\newblock {\"U}ber die {V}erteilung der {W}urzeln bei gewissen algebraischen
  {G}leichungen mit ganzzahligen {K}oeffizienten.
\newblock {\em Math. Zeitschr.}, 17:19--22, 1923.

\bibitem{MF2}
M.~Fekete.
\newblock {\"U}ber den transfiniten {D}urchmesser ebener {P}unktmengen.
\newblock {\em Acta Litterarum ac Scientiarum Regiae Universitatis Hungaricae
  Francisco-Josephinae}, 5:19--22, 1930.

\bibitem{G2014b}
O.~Gu\'edon.
\newblock Concentration phenomena in high dimensional geometry.
\newblock In {\em Journ\'ees {MAS} 2012}, volume~44 of {\em ESAIM Proc.}, pages
  47--60. EDP Sci., Les Ulis, 2014.

\bibitem{G2014a}
O.~Gu\'edon, P.~Nayar, and T.~Tkocz.
\newblock Concentration inequalities and geometry of convex bodies.
\newblock In {\em Analytical and probabilistic methods in the geometry of
  convex bodies}, volume~2 of {\em IMPAN Lect. Notes}, pages 9--86. Polish
  Acad. Sci. Inst. Math., Warsaw, 2014.

\bibitem{GP2007}
O.~Gu\'edon and G.~Paouris.
\newblock Concentration of mass on the {S}chatten classes.
\newblock {\em Ann. Inst. H. Poincar\'e Probab. Statist.}, 43(1):87--99, 2007.

\bibitem{HardyLDP}
A.~Hardy.
\newblock A note on large deviations for 2{D} {C}oulomb gas with weakly
  confining potential.
\newblock {\em Electron. Commun. Probab.}, 17:no. 19, 12, 2012.

\bibitem{hiai_petz}
F.~Hiai and D.~Petz.
\newblock {\em The semicircle law, free random variables and entropy},
  volume~77 of {\em Mathematical Surveys and Monographs}.
\newblock American Mathematical Society, Providence, RI, 2000.

\bibitem{HPV17}
A.~Hinrichs, J.~Prochno, and J.~Vyb\'{i}ral.
\newblock Entropy numbers of embeddings of {S}chatten classes.
\newblock {\em J. Funct. Anal.}, 273(10):3241 -- 3261, 2017.

\bibitem{KPT17CCM}
Z.~Kabluchko, J.~Prochno, and C.~Th\"ale.
\newblock High-dimensional limit theorems for random vectors in
  $\ell_p^n$-balls.
\newblock {\em Commun. Contemp. Math. (online ready)}, 2017.

\bibitem{Kallenberg}
O.~Kallenberg.
\newblock {\em Foundations of {M}odern {P}robability}.
\newblock Probability and its Applications. Springer-Verlag, New York, second
  edition, 2002.

\bibitem{KallenbergRM}
O.~Kallenberg.
\newblock {\em Random {M}easures, {T}heory and {A}pplications}, volume~77 of
  {\em Probability {T}heory and {S}tochastic {M}odelling}.
\newblock Springer, Cham, 2017.

\bibitem{KMP1998}
H.~K\"onig, M.~Meyer, and A.~Pajor.
\newblock The isotropy constants of the {S}chatten classes are bounded.
\newblock {\em Math. Ann.}, 312(4):773--783, 1998.

\bibitem{koroljuk_book}
V.S. Koroljuk and Y.V. Borovskich.
\newblock {\em Theory of $U$-statistics}, volume 273 of {\em Mathematics and
  its Applications}.
\newblock Kluwer Academic Publishers, Dordrecht, The Netherlands, 1994.

\bibitem{mhaskar_saff}
H.N. Mhaskar and E.B.~\ Saff.
\newblock Extremal problems for polynomials with exponential weights.
\newblock {\em Trans. Amer. Math. Soc.}, 285(1):203--234, 1984.

\bibitem{NaorTAMS}
A.~Naor.
\newblock The surface measure and cone measure on the sphere of {$\ell_p^n$}.
\newblock {\em Trans. Amer. Math. Soc.}, 359(3):1045--1079 (electronic), 2007.

\bibitem{NR2003}
A.~Naor and D.~Romik.
\newblock Projecting the surface measure of the sphere of {$\ell_p^n$}.
\newblock {\em Ann. Inst. H. Poincar\'e Probab. Statist.}, 39(2):241--261,
  2003.

\bibitem{PS2011}
L.~Pastur and M.~Shcherbina.
\newblock {\em Eigenvalue distribution of large random matrices}, volume 171 of
  {\em Mathematical Surveys and Monographs}.
\newblock American Mathematical Society, Providence, RI, 2011.

\bibitem{PS1931}
G.~P\'olya and G.~Szeg\"o.
\newblock {\"U}ber den transfiniten {D}urchmesser ({K}apazit\"atskonstante) von
  ebenen und r\"aumlichen {P}unktmengen.
\newblock {\em J. Reine Angew. Math.}, 165:4--49, 1931.

\bibitem{RV2016}
J.~{Radke} and B.-H. {Vritsiou}.
\newblock {On the thin-shell conjecture for the {S}chatten classes}.
\newblock {\em ArXiv e-prints}, February 2016.

\bibitem{rakhmanov}
E.A. Rakhmanov.
\newblock Asymptotic properties of orthogonal polynomials on the real axis.
\newblock {\em Mat. Sb. (N.S.)}, 119(161)(2):163--203, 303, 1982.

\bibitem{SaffBOOK}
E.B. Saff and V.~Totik.
\newblock {\em Logarithmic {P}otentials with {E}xternal {F}ields}, volume 316
  of {\em Grundlehren der Mathematischen Wissenschaften [Fundamental Principles
  of Mathematical Sciences]}.
\newblock Springer-Verlag, Berlin, 1997.
\newblock Appendix B by Thomas Bloom.

\bibitem{R1984}
J.~Saint~Raymond.
\newblock Le volume des id\'eaux d'op\'erateurs classiques.
\newblock {\em Studia Math.}, 80(1):63--75, 1984.

\bibitem{SchechtmanSchmuckenschlaeger}
G.~Schechtman and M.~Schmuckenschl\"ager.
\newblock Another remark on the volume of the intersection of two {$L^n_p$}
  balls.
\newblock In {\em Geometric aspects of functional analysis (1989--90)}, volume
  1469 of {\em Lecture Notes in Math.}, pages 174--178. Springer, Berlin, 1991.

\bibitem{SchechtmanZinn}
G.~Schechtman and J.~Zinn.
\newblock On the volume of the intersection of two {$L^n_p$} balls.
\newblock {\em Proc. Amer. Math. Soc.}, 110(1):217--224, 1990.

\bibitem{schmuckenschlaeger_pams}
M.~Schmuckenschl\"ager.
\newblock Volume of intersections and sections of the unit ball of {$l^n_p$}.
\newblock {\em Proc. Amer. Math. Soc.}, 126(5):1527--1530, 1998.

\bibitem{SchmuckenschlaegerCLT}
M.~Schmuckenschl\"ager.
\newblock C{LT} and the volume of intersections of {$l^n_p$}-balls.
\newblock {\em Geom. Dedicata}, 85(1-3):189--195, 2001.

\bibitem{ST1980}
S.~Szarek and N.~Tomczak-Jaegermann.
\newblock On nearly euclidean decomposition for some classes of {B}anach
  spaces.
\newblock {\em Compositio Math.}, 40(3):367--385, 1980.

\bibitem{VA1987}
W.~Van~Assche.
\newblock {\em Asymptotics for {O}rthogonal {P}olynomials}, volume 1265 of {\em
  Lecture Notes in Mathematics}.
\newblock Springer-Verlag, Berlin, 1987.

\end{thebibliography}

\end{document}